\documentclass{amsart}
\usepackage{amssymb,epsfig,amsxtra, amsmath}
\usepackage[mathscr]{eucal}
\usepackage{a4}
\input{xy}
\xyoption{all}
\newtheorem{theorem}{\textbf{Theorem}}[section]

\newtheorem{definition}[theorem]{\textbf{Definition}}
\newtheorem{lemma}[theorem]{\textbf{Lemma}}
\newtheorem{corollary}[theorem]{\textbf{Corollary}}
\newtheorem{remark}[theorem]{\textbf{Remark}}

\newtheorem{problem}[theorem]{\textbf{Problem}}

\def\0{|\mathcal O_{S^{d-1}\cup\tilde\pi}(n)|}

\begin{document}
\title[]{Degenerating curves and surfaces: first results}
\author{Concettina Galati }

\address{via P. Bucci, cubo 30B
87036 Arcavacata di Rende (CS), Italy. }

\email{galati@mat.unical.it}

\thanks{}



\date{}

\dedicatory{}

\commby{}


\begin{abstract}
Let $\mathcal S\to\mathbb A^1$ be a smooth family of surfaces whose
general fibre is a smooth surface of $\mathbb P^3$ and whose special
fibre has two smooth components, intersecting transversally along a
smooth curve $R$. We consider the Universal Severi-Enriques variety
$\mathcal V$ on $\mathcal S\to\mathbb A^1$. The general fibre of
$\mathcal V$ is the variety of curves on $\mathcal S_t$ in the
linear system $|\mathcal O_{\mathcal S_t}(n)|$ with $k$ cusps and
$\delta$ nodes as singularities. Our problem is to find all
irreducible components of the special fibre of $\mathcal V$. In this
paper, we consider only the cases $(k,\delta)=(0,1)$ and
$(k,\delta)=(1,0)$. In particular, we determine all singular curves
on the special fibre of $\mathcal S$ which, counted with the right
multiplicity, are a limit of $1$-cuspidal curves on the general
fibre of $\mathcal S$.
\end{abstract}


\maketitle
\section{Introduction}\label{introduction}
In this section we introduce our problem and we fix notation. Let
$\mathcal F$ be the pencil of surfaces of the complex projective
space $\mathbb P^3(\mathbb C):=\mathbb P^3$, generated by a general
surface $S^d$ of degree $d\geq 2$ and a reducible surface
$S_0=S^{d-1}\cup \pi$, where $S^{d-1}$ is a general surface of
degree $d-1$ and $\pi$ is a general plane, intersecting $S^{d-1}$
along a smooth curve $R$. The singular locus of the total space
$\mathcal A$ of $\mathcal F$ consists of the $d(d-1)$ points
$p_1,\dots,p_{d(d-1)}$ of the special fibre which are intersection
of $S^d$ and $R=S^{d-1}\cap \pi$. Moreover, at every such a point,
$\mathcal A$ has a rational double singularity. Let $\mathcal
S\to\mathbb A^1$ be the smooth family of surfaces obtained by
smoothing $\mathcal A$ and by contracting the exceptional components
on $\pi$. The general fibre $\mathcal S_t$ of $\mathcal S$ is
isomorphic to the general fibre of $\mathcal A$, while the special
fibre of $\mathcal S$ is $\mathcal S_0=A\cup B$ where $A=S^{d-1}$
and $B$ is the blowing-up of $\pi$ at the $d(d-1)$ singular points
of $\mathcal A$. In particular, we have that
$R^2_{B}=(d-1)^2-d(d-1)=-(d-1)=-R^2_{A}$. From now on, we shall
indicate by $E_1,\dots,\,E_{d(d-1)}$ the exceptional curves of $B$.
Now, by denoting by $H$ the pull-back to $\mathcal S$ of the
hyperplane divisor of $\mathbb P^3$, by $H_t$ the restriction of $H$
to the fibre $\mathcal S_t$ of $\mathcal S$ and by $p_a(d,n)$ the
arithmetic genus of a divisor in $|\mathcal O_{\mathcal
S_t}(nH_t)|$, we consider the locally closed set, in the Zariski
topology, $\mathcal W^\mathcal S_{nH,k,\delta} \subset |\mathcal
O_{\mathcal S}(nH)|\times (\mathbb A^1\setminus\{0\})$ defined by
\begin{eqnarray*}
\mathcal W^\mathcal S_{nH,k,\delta}=\{([D],t)\mid\,\mathcal
S_t\,\textrm{is smooth,}\, D\cap \mathcal S_t:=D_t\in\,\mid\mathcal
O_{\mathcal S_t}(nH_t)\mid\, \textrm{is
irreducible}\\
\textrm{of genus}\,g= p_a(d,n)-k-\delta \textrm{ having}\, \delta\,
\textrm{nodes and}\, k\, \textrm{cusps as singularities}.\}
\end{eqnarray*}
From now on, we will often denote by $|\mathcal O_{\mathcal
S_t}(n)|$ the linear series $|\mathcal O_{\mathcal S_t}(nH_t)|$.
Now, if $\mathcal H_{\mathcal S|\mathbb A^1}$ is the relative
Hilbert Scheme of the family $\mathcal S\to\mathbb A^1$ and if
$\mathcal H_n$ is the irreducible component of $\mathcal H_{\mathcal
S|\mathbb A^1}$ whose fibre over $t$, with respect to the natural
morphism $\mathcal H_{\mathcal S|\mathbb A^1}\to\mathbb A^1$, is
$|\mathcal O_{\mathcal S_t}(nH_t)|$, one defines a natural rational
map
$$f:\mathcal W^\mathcal S_{nH,k,\delta}\dashrightarrow\mathcal H_n\subset\mathcal
H_{\mathcal S|\mathbb A^1},$$ sending the point $([D],t)$ to the
point $[D_t]\in \mathcal H_n$, parametrizing the curve $D_t=D\cap
\mathcal S_t.$ We denote by
$$\mathcal V^\mathcal S_{nH,k,\delta}\subset\mathcal H_n$$
the closure in the Zariski topology of the image of $\mathcal
W^\mathcal S_{nH,k,\delta}$ in $\mathcal H_n$, with respect to $f$.
Since the general fibre $\mathcal V_t$ of $\mathcal
V_{nH,k,\delta}^{\mathcal S}$ is the Severi-Enriques variety
$\mathcal V_t=\mathcal V^{\mathcal S_t}_{nH_t,k,\delta}$ of
irreducible curves on $\mathcal S_t$ in the linear system $|nH_t|$,
with $\delta$ nodes and $k$ cusps as singularities, we call
$\mathcal V^\mathcal S_{nH,k,\delta}$ \textit{the Universal
Severi-Enriques Variety on} $\mathcal S$ of irreducible curves in
the linear system $|nH|$, with $\delta$ nodes and $k$ cusps. The
problem we are interested is the following.
\begin{problem}[Main Problem]\label{main} Let $\mathcal
V_0$ be the special fibre of $\mathcal V_{nH,k,\delta}^\mathcal S$.
Then every its irreducible component $ V_0^i$ will have dimension
equal to $\dim{\mathcal V_{nH,k,\delta}^\mathcal S}-1=\dim(\mathcal
V_t).$ We want to determinate all irreducible components $V_0^i$ of
$\mathcal V_0=\sum_i m_i V_0^i$, with their respective multiplicity
$m_i$.
\end{problem}
This problem has been already studied by Ran in \cite{ran},
\cite{ran1} and \cite{ran2} in the case $k=0$ and $\mathcal S$ equal
to a family whose general fibre $\mathcal S_t$ is a plane and whose
special fibre is $\mathcal S^0=A\cup B$, where $A\simeq\mathbb P^2$
and $B\simeq\mathbb F_1$. Our paper is strongly inspired to Ran's
works but our approach is very different. The techniques we use to
find all possible irreducible components of the special fibre of the
Universal Enriques-Severi variety have been introduced by Ciliberto
e Miranda in \cite{cm} and the aim of this paper is also to show how
the notion of limit linear system can be useful for studying the
problem \ref{main}. Unfortunately, if these kinds of methods are
very helpful in order to determine all irreducible components of
$\mathcal V_0\subset\mathcal V_{nH,k,\delta}^{\mathcal S}$, they are
not enough to find their respective multiplicity. Here we consider
only the cases $(k,\delta)=(0,1)$ and $(k,\delta)=(1,0)$. In both
these cases, by only using Ciliberto and Miranda techniques, we are
able to describe all irreducible components of $\mathcal V_0$ and to
compute their respective "geometric multiplicity", which we are
going to define.
\begin{definition}\label{geometricmultiplicity}
Let $V$ be an irreducible component of the special fibre $\mathcal
V_0$ of the Universal Severi-Enriques variety $\mathcal
V_{nH,k,\delta}^\mathcal S$. The geometric multiplicity of $V$ in
$\mathcal V_0$ is the minimal integer $m$ such that there exist
local analytic $m$-multisections passing through the general element
$[D]$ of $V$ and intersecting the general fibre $\mathcal V_t$ of
$\mathcal V_{nH,k,\delta}^\mathcal S$ at $m$ general points.
\end{definition}
\begin{remark}
Notice that the geometric multiplicity of an irreducible component
$V_0^i$ of $\mathcal V_0$ coincides with the usual multiplicity if
the Universal Severi-Enriques Variety is smooth at the general
element of $V_0^i$. When $(k,\delta)=(1,0)$ or $(k,\delta)=(0,1)$,
the geometric multiplicity of every irreducible $V$ of $\mathcal
V_0$ is in fact the multiplicity of $V$ in $\mathcal V_0$ and, by
counting every irreducible component with multiplicity equal to
geometric multiplicity, we can obtain a recursion formula for the
degree of $\mathcal V_{nH,1,0}^{\mathcal S_t}$ and $\mathcal
V_{nH,0,1}^{\mathcal S_t}$ for every $d\geq 2$. Actually, the well
known formulas of the number of one-nodal curves in a pencil and
one-cuspidal curves in a net can be used to prove that in these
cases geometric multiplicity and multiplicity coincide. To compute
explicitly the tangent space of the Universal Severi-Enriques
Variety at the general point of every its irreducible component is
not trivial, even in the cases $(k,\delta)=(0,1)$ and
$(k,\delta)=(1,0)$ and it requires very different methods from those
we use in this paper. For this reason, we prefer to approach this
problem in a future and more general paper. We just want to say that
geometric multiplicity and multiplicity may be different even for
families of curves with a very few singularities. For example, the
special fibre $\mathcal V_0$ of $\mathcal V_{nH,0,2}$, has an
component (parametrizing curves having two "simple tacnodes " and
nodes at the intersection points with the singular locus $R$ of
$\mathcal S_0$) having, by our computations, geometric multiplicity
$2$ but multiplicity $4$, according to Ran's results.
\end{remark}
Section \ref{one-nodal curves} is devoted to the case
$(k,\delta)=(0,1)$ whereas section \ref{one cuspidal curves} is
devoted to the case $(k,\delta)=(1,0)$. We want to stress that, even
if our proof is different, the results of section \ref{one-nodal
curves} are already known and can be deduced, with some further
observations, by theorems 2.1 and 2.2 of \cite{chenversal}. On the
contrary, as far as we know, the analysis we do of the special fibre
$\mathcal V_0$ of $\mathcal V_{nH,1,0}^\mathcal S$ is completely
new. Techniques we use in this paper can be applied to any smooth
family of surfaces whose special fibre is irreducible and also to
families of curves with singularities different from nodes or cusps.
We conclude this section by introducing some terminology.
\subsection{Terminology} Let $\mathcal X\to \mathbb A^1$
be any family of surfaces obtained from $\mathcal S\to\mathbb A^1$
by a finite number of base changes and blowing-ups and let
$f:\mathcal X\to\mathbb P^3$ be the natural morphism from $\mathcal
X$ to $\mathbb P^3$. By abusing notation, we will denote always by
$H$ the pull-back to $\mathcal X$ of the hyperplane divisor class of
$\mathbb P^3$. Similarly, if $\mathcal Y\to\mathbb A^1$ is a family
of surfaces obtained from $\mathcal X\to\mathbb A^1$ by a finite
number of base changes and blow-ups and $L\subset\mathcal X_0$ is a
curve lying in an irreducible component of the special fibre of
$\mathcal X$, then we will usually denote by the same symbol $L$ the
proper transform of $L\subset\mathcal X$ in $\mathcal Y$. Moreover,
if $S\subset\mathbb P^3$ is a surface of degree $n$ and
$C_t=f^*S\cap \mathcal X_t$ is the curve cut out on the fibre
$\mathcal X_t$ of $\mathcal X$ by the pull-back of $S$ via $f$, from
now on we will say that \textit{$C_t$ is cut out on $\mathcal X_t$
by $S$.} Furthermore, if $W\subset |\mathcal O_A(n)|$ and $V\subset
|\mathcal O_B(n)|$ are two projective subvarieties and $\mathcal
W\subset H^0(A,\mathcal O_A(n))$ and $\mathcal V\subset
H^0(B,\mathcal O_B(n))$ are the affine cones associated to $W$ and
$V$ respectively, then, by abusing notation, \textit{we shall denote
by} $V\times_{|\mathcal O_R(n)|}W$ \textit{the projective variety}
$\mathbb P(\mathcal W\times_{H^0(R,\mathcal O_R(n))}\mathcal
V)\subset |\mathcal O_{A\cup B}(n)|.$ Finally, we will indicate by
$\mathbb F_n$ the Hirzebruch surface $\mathbb F_n=\mathbb P(\mathcal
O_{\mathbb P^1}\oplus\mathcal O_{\mathbb P^1}(n)).$

\section{The case of one-nodal curves}\label{one-nodal curves}
In this section we consider the case $k=0$ and $\delta=1$,
determining all irreducible components of $\mathcal V_0$ with the
respective geometric multiplicity (see definition
\ref{geometricmultiplicity}). Notice that, in this case, the Severi
variety $\mathcal V_t=\mathcal V^{\mathcal S^t}_{nH_t,0,1}$ on the
general fibre is irreducible, it has the expected dimension, it is
smooth at every point corresponding to a one-nodal curve and,
finally, it coincides with the locus of $|nH_t|$ parametrizing
singular curves. Moreover, before we describe the special fibre
$\mathcal V_0$ of $\mathcal V_{nH,0,1}^\mathcal S$, notice that the
restriction of the linear system $|nH|$ to the special fibre
$\mathcal S_0=A\cup B$ of $\mathcal S$ is
$$
|\mathcal O_{A\cup B}(n)|\backsimeq \mathbb P(H^0(A,\mathcal
O_{A}(n))\times_{H^0(R,\mathcal O_{R}(n))}H^0(B,\mathcal O_{B}(n))),
$$
and, by Bertini Theorem, the general element of $|\mathcal O_{A\cup
B}(n)|$ is a curve smooth outside $R$ and with nodes at its
intersection points with $R$. By using notation introduced in the
previous section, we also denote by $W_A(d,n)$ and $W_{B}(d,n)$ the
divisors of $|\mathcal O_{A\cup B}(n)|$ defined as
$$W_A(d,n):= V^{A}_{nH,0,1}\times_{|\mathcal
O_{R}(n)|}|\mathcal O_{B}(n)|\,\,\textrm{and}\,\,W_{B}(d,n):=
|\mathcal O_{A}(n)|\times_{|\mathcal O_{R}(n)|} V^{B}_{nH,0,1}.
$$
Notice that the general element of $W_A(d,n)$ corresponds to a curve
cut out on $A\cup B\subset \mathbb P^3$ by a surface of degree $n$
tangent to $A$ at a general point. Moreover, if $(d,n)\neq (2,1)$,
$(d,n)\neq (3,1)$, $d\geq 2$ and $n\geq 1$ then the general element
of $W_A(d,n)$ corresponds to a curve $D=D_A\cup D_{B}$, where $D_A$
is an irreducible one-nodal curve intersecting transversally $R$ and
$D_{B}$ is a smooth curve on $B$, not intersecting $E_i$, for every
$i$, and intersecting transversally $R$ at the points $A\cap R$.
Similarly, for $d,\,n\geq 2$ the general element of $W_{B}(d,n)$
corresponds to a curve $D=D_A\cup D_{B}$ such that $D_A\subset A$ is
smooth, intersecting transversally $R$, and $D_{B}$ is a one-nodal
curve (which is irreducible for $(d,n)\neq (d,3)$), not intersecting
$E_i$ for every $i$, and such that $D_{B}\cap R=D_A\cap R$. In
particular, $W_A(d,n)$ and $W_{B}(d,n)$ are both irreducible
divisors of $|\mathcal O_{A\cup B}(n)|$. Finally, we denote by
$W_{E_i}(d,n)$, $i=1,\dots,\,d(d-1)$ and $T(d,n)$ the irreducible
divisors of $|\mathcal O_{A\cup B}(n)|$ defined as
$$
W_{E_i}(d,n)=\{[D]\in |\mathcal O_{A\cup B}(n)|\,\, \textrm{such
that}\,\,E_i\subset D\},
$$
where $E_1,\dots,\,E_{d(d-1)}$ are the exceptional divisors of $B$,
and
\begin{eqnarray*} T(d,n)& = &\{[D=D_A\cup
D_{B}]\in |\mathcal O_{A\cup B}(n)|,\,\,\textrm{such that}\,\,D_A\,\,
\textrm{and}\,\,D_{B}\,\,\textrm{are tangent to} \\
& & R\,\,\textrm{at a general point}\}.
\end{eqnarray*}
By using the terminology introduced at the end of the previous
section, notice that $W_{E_i}(d,n)$ parametrizes curves on $A\cup B$
cut out by surfaces of degree $n$ in $\mathbb P^3$ passing through
$p_i$, for every $i=1,\dots,d(d-1)$, whereas $T(d,n)$ parametrizes
curves on $A\cup B$ cut out by surfaces of degree $n$ of $\mathbb
P^3$ tangent to $R$ at its general points. In particular, if
$[D=D_A\cup D_{B}]\in T(d,n)$ is a general point, then $D_A$ and
$D_{B}$ are smooth, tangent to $R$ at only one point and they will
be transverse to $R$ outside this point, except for the case
$(d,n)\neq (2,1)$, when $D_A=D_B=R$.

\begin{lemma}\label{uno}
$W_A(d,n)$ is an irreducible component of $\mathcal V_0$ of
geometric multiplicity $1$, for every $(d,n)\neq (2,1)$. Similarly,
$W_{B}(d,n)$ is an irreducible component of $\mathcal V_0$ of
geometric multiplicity $1$, if $(d,n)\geq (2,2)$. Moreover, if $V$
is an irreducible component of $\mathcal V_0$ of geometric
multiplicity $1$ and whose general element corresponds to a curve
$D=D_A\cup D_{B}$ which does not contain any $E_i\subset B$, then
$V$ coincides with $W_A(d,n)$ or $W_{B}(d,n)$.
\end{lemma}
\begin{proof}
Assume that $(d,n)\neq (2,1)$, let $\gamma$ be a section of
$\mathcal S\to\mathbb A^1$ passing through a general point $p$ of
$A$ and let $S\sim nH$ be a general divisor singular along $\gamma$.
Notice that such a divisor exists and, by generality, its general
fibre $S\cap \mathcal S_t$ is a $1$-nodal curve. Let $\mathcal
X=Bl_\gamma\mathcal S$ with exceptional divisor $\Gamma$. Then, if
$S^\prime$ is the proper transform of $S$ on $\mathcal X$, we have
that $S^\prime\sim nH-2\Gamma$. Now, denoting by $\mathcal X_t$ the
fibre over $t$ of $\mathcal X$, consider the exact sequence
$$
0\to\mathcal O_{\mathcal X}(nH-2\Gamma-\mathcal X_t)\to\mathcal
O_{\mathcal X}(nH-2\Gamma)\to\mathcal O_{\mathcal
X_t}(nH-2\Gamma)\to 0.
$$
Since the fibres of the family $\mathcal X\to\mathbb A^1$ are
linearly equivalent, we have that the dimension of the image of the
map $\xymatrix@1{H^0(\mathcal X,\,O_{\mathcal
X}(nH-2\Gamma))\ar[r]^{r_t}& H^0(\mathcal X_t,\,O_{\mathcal
X_t}(nH-2\Gamma))}$ doesn't depend on $t$. Moreover, since for $t$
general the map $r_t$ is surjective, we have that
$dim(Im(r_0))=dim(Im(r_t))=h^0(\mathcal X_t,\mathcal O_{\mathcal
X_t}(n H-2\Gamma))=h^0(\mathcal X_t,\mathcal O_{\mathcal X_t}(n
H))-3=h^0(\mathcal X_0,\mathcal O_{\mathcal X _0}(n H-2\Gamma))$.
So, every curve of $A\cup B$ with a node at the point $\gamma\cap A$
is a limit of a one-nodal curve in the linear system $|nH|$ on
$\mathcal S_t$. By the generality of the point $\gamma\cap A$ on $A$
we have that $W_A(d,n)$ is an irreducible component of $\mathcal
V_0$ of geometric multiplicity $1$. In order to prove the lemma for
$W_{B}(d,n)$ repeat the same argument as used for $W_{A}(d,n)$.

Now, let $V$ be an irreducible component of $\mathcal V_0$, whose
general element $[D]$ corresponds to a curve $D=D_A\cup D_{B}$ which
does not contain any $E_i\subset B$ and such that there exists an
analytic neighborhood $U$ of $0$ and an analytic section $\Delta$ of
the universal Severi variety $\mathcal V^\mathcal S_{nH,0,1}|_U$,
passing through $[D]$. Then, the singular locus of the family of
curves $\mathcal C\to\Delta$, naturally parametrized by $\Delta$, is
a section of $\mathcal S|_U\to\Delta$ and it must intersect the
special fibre at a smooth point $q$, i.e. at a point $q\notin
R=A\cap B$. If $q\in A$ then $D_A$ is singular, $[D]\in W_A(d,n)$
and $V=W_A(d,n)$. Otherwise $V=W_{B}(d,n)$.
\end{proof}
\begin{remark}
By the previous lemma, the other possible irreducible components of
$\mathcal V_0$ "are produced" by degenerations of the general
element $[D_t]$ of $\mathcal V_t$ such that as $D_t$ goes to $A\cup
B$, the node of $D_t$ specializes to a point of the singular locus
$R$ of $A\cup B$.
\end{remark}
\begin{lemma}\label{tacnodo}
If $(d,n)\neq (2,1)$, then $T(d,n)$ is an irreducible component of
$\mathcal V_0$ of geometric multiplicity $2$.
\end{lemma}
\begin{proof}
Let $p\in R=A\cup B=\mathcal S_0$ be a point such that $p\notin
E_i$, for every $i$. We consider a double covering of our family
$\mathcal S\to\mathbb A^1$ totally ramified at its special fibre
\begin{displaymath}
\xymatrix{ \mathcal S^\prime \ar[d]\ar[r]^h &
\mathcal S\ar[d]\\
\mathbb A^1\ar[r]^{\nu_2}&\mathbb A^1}
\end{displaymath}
Now, the special fibre $\mathcal S^\prime_0$ of $\mathcal
S^\prime\to \mathbb A^1$ is isomorphic to the special fibre
$\mathcal S_0$ of $\mathcal S$ and we still denote it by $A\cup B$.
But, if $xy=t$ is the local equation of $\mathcal S$ at $p$, the
local equation of $\mathcal S^\prime$ at $p$ will be $xy=t^2$. In
particular, $\mathcal S^\prime$ is singular along the singular locus
$R=A\cap B$ of the special fibre. If we blow-up the ambient space
along $R$, the proper transform $\mathcal X$ of $\mathcal S^\prime$
is smooth, the general fibre of $\mathcal X$ is isomorphic to the
general fibre of $\mathcal S^\prime$ while the special fibre of
$\mathcal X$ is $\mathcal X_0=A\cup \mathcal E\cup B$, where
$\mathcal E\backsimeq\mathbb F_{d-1}$ is the intersection of
$\mathcal X$ with the exceptional divisor.
The inverse image of $p$ to $\mathcal X$ is a fibre of $\mathcal E$
which we denote by $F$. Now, if $U$ is an analytic neighborhood of
$0\in \mathbb A^1$ small enough, let $\gamma$ be a section of the
family $\mathcal X|_U$ and let $\mathcal C\subset\mathcal X|_U$ be a
general divisor with nodal singularities along $\gamma$ and linearly
equivalent to $nH$, where we still denote by $H$ the pull-back to
$\mathcal X$  of the hyperplane divisor. Notice that such a divisor
$\mathcal C$ exists for every $n\geq 1$ and $d\geq 2$, because, for
every $n\geq 1$ and $d\geq 2$, on the general fibre $\mathcal X_t$
of $\mathcal X$, there exist curves, linearly equivalent to
$nH|_{\mathcal X_t} =n H_t$ with only a node at a general point
$q_t$ and no further singularities. We want to understand the kind
of singularities of the special fibre $\mathcal C_0$ of $\mathcal
C$. To this aim, let $\pi :\mathcal Y\to\mathcal X|_{U}$ be the
blowing-up of $\mathcal X|_{U}$ along $\gamma$, with exceptional
divisor $\Gamma$. The special fibre of $\mathcal Y$ is now given by
$A\cup \mathcal E^\prime\cup B$, where $\mathcal E^\prime$ is the
blowing-up of $\mathcal E$ at $\gamma\cap F=\gamma\cap \mathcal E$,
with new exceptional divisor $E_0=\Gamma\cap \mathcal E^\prime$. In
particular, $F^2_{\mathcal E^\prime}=-1$ and, if $\mathcal C^\prime$
is the proper transform of $\mathcal C$ on $\mathcal Y$, then
$\mathcal C^\prime\sim nH-2\Gamma$.

Now, if $(d,n)=(2,1)$, denoting by $F_{\mathcal E^\prime}$ the
pull-back to $\mathcal E^\prime$ of the linearly equivalence class
of the fibre of $\mathcal E$, we have that the divisor
$(nH-2\Gamma)|_{\mathcal E^\prime}\sim F_{\mathcal E^\prime}-2E_0$
is not effective. Since $\mathcal C^\prime$ is effective we deduce
that $\mathcal E^\prime\subset \mathcal C^\prime$ and, in
particular, $R_1\subset\mathcal C^\prime|_A$ and $R_2\subset
C^\prime|_B$. So the point $[C]\in |\mathcal O_{A\cup B}(n)|$,
corresponding to the curve $C_A\cup C_B$, where $C_A=\mathcal C|_A$
and $C_B=\mathcal C|_B$, belongs to the variety $T(2,1)$ and cannot
be general in any irreducible component of $\mathcal V_0$.

If $(d,n)\neq (2,1)$, then the divisor $(nH-2\Gamma)|_{\mathcal
E^\prime}\sim (n(d-1)-2)F_{\mathcal E^\prime}+2F$. In particular,
$F$ is contained in the divisor $(nH-2\Gamma)|_{\mathcal E^\prime}$
with multiplicity at least $2$ and $\mathcal C^\prime\cdot F=-2$. We
want to compute the multiplicity $\alpha$ of $\mathcal C^\prime$
along $F$. Let then $\pi^\prime :\mathcal Z\to\mathcal Y$ be the
blowing-up of $\mathcal Y$ along $F$. The special fibre of $\mathcal
Z$ is $A^\prime\cup \mathcal E^\prime\cup B^{\prime}\cup\Theta$,
where $\Theta$ is the new exceptional divisor and $A^\prime$ and
$B^{\prime}$ are the blowups of $A$ and $B$ at the point $F\cap A$
and $F\cap B$, with exceptional divisor $\Theta\cap A^\prime$ and
$\Theta\cap B^{\prime}$ respectively. By the triple point formula
$$
\Theta\cdot \mathcal E^\prime\cdot(A^\prime+ \mathcal E^\prime+
B^{\prime} +\Theta)=0
$$
we deduce that $(F^2)_{\Theta}=-(F^2)_{\mathcal E^\prime}-2=-1$ and,
in particular, $\Theta\simeq\mathbb F_1$ and $F=\mathcal
E^\prime\cap \Theta$ is the exceptional divisor of $\Theta$.
Moreover, denoting by $\mathcal C^{\prime\prime}$ the pullback to
$\mathcal Z$ of $\mathcal C^\prime$ and by $F_\Theta$ the linear
equivalence class of the fibre of $\Theta$ , we have that
$$\mathcal C^{\prime\prime} |_{\Theta}\sim (2\alpha-2)F_\Theta+\alpha
F.$$ Since $\mathcal C^{\prime\prime} |_{\Theta}$ must be an
effective divisor, we have that $2\alpha-2\geq 0$, i.e. $\alpha\geq
1$. Now, for every $\alpha\geq 1$, we have that the image into
$\mathcal X|_U$ of a divisor in $\mathcal Z$, linearly equivalent to
$nH-2\Gamma^\prime-\alpha\Theta$, is a divisor linearly equivalent
to $nH$ and with double singularities along $\gamma$. Since we have
taken a general divisor $\mathcal C$ in $\mathcal X|_U$, with these
properties, we may assume that $\alpha$ is the minimum integer in
order that $\mathcal C^{\prime\prime}|_\Theta$ is effective, i.e.
$\alpha=mult_F(\mathcal C^\prime)=1$ and $\mathcal
C^{\prime\prime}|_\Theta =F=C^{\prime\prime}|_{\mathcal E^\prime}$.
This implies that $\mathcal C^{\prime\prime}|_{A^\prime}$ must
intersect the exceptional divisor $\Theta \cap A^\prime$ with
multiplicity one at the points $F\cap A$. Thus, recontracting
$\Theta$ and going back to $\mathcal C^\prime\subset\mathcal Y$, we
have that $\mathcal C^\prime|_{A}$ must pass through the point
$F\cap R_1$ and it must be smooth and tangent to $R_1$ at this
point. At the same way, $\mathcal C^\prime|_{B}$ passes through
$F\cap R_2$ with multiplicity $1$ and it is tangent to $R_2$ at this
point.

Now, let $\mathcal D$ be the image of $\mathcal C$ into $\mathcal
S|_{\nu_2(U)}$. If $t$ is a general point of $\mathbb A^1$ and
$\{t_1, \,t_2\}=\nu_2^{-1}(t)$, then the fibre of $\mathcal D$ over
$t$ is $\mathcal D_t=\mathcal C_{t_1}\cup\mathcal C_{t_2}$ and, in
particular, it is the union of two one-nodal curves in the linear
system $|nH_t|$. Whereas, the special fibre $\mathcal D_0$ of
$\mathcal D$ is the curve, counted with multiplicity $2$, image of
$\mathcal C_0$ under the contraction of $\mathcal E$. Thus $\mathcal
D_0=2\mathcal C_A\cup 2\mathcal C_{B}$, where $\mathcal C_A=\mathcal
C\cap A$ and $\mathcal C_{B}=\mathcal C\cap B$. \textit{From what we
proved before, the point $x=[\mathcal C_A\cup \mathcal C_{B}]\in
T(d,n)\cap\mathcal V_0$ and the curve in $\mathcal H_n$
corresponding to $\mathcal D$ is a local bisection of $\mathcal
V_{nH,0,1}^{\mathcal S}$ passing through $[D]$ and intersecting the
general fibre at $2$ general points.}

\textit{Now we want to prove that the point $x$ is general in
$T(d,n)$}. To this aim let $\mathcal Z^1$ be the blowing-up of
$\mathcal Z$ along $F=\Theta\cap \mathcal E^\prime$. Now the special
fibre of $\mathcal Z^1$ is $\mathcal Z_0^1=A^{\prime\prime}\cup
\mathcal E^\prime\cup B^{\prime\prime}\cup \Theta \cup 2\Theta_1$,
where $\Theta_1\simeq\mathbb F_0$ is the new exceptional divisor and
$A^{\prime\prime}$ and $B^{\prime\prime}$ are the proper transforms
of $A^{\prime}$ and $B^{\prime}$ respectively. Moreover, denoting by
$\mathcal C^{\prime\prime\prime}$ the proper transform of $\mathcal
C^{\prime\prime}$ on $\mathcal Z_1$, we have that
$$
\mathcal C^{\prime\prime\prime}\sim
nH-2\Gamma^{\prime\prime}-\Theta-2\Theta_1,$$ where
$\Gamma^{\prime\prime}$ is the proper transform of $\Gamma^\prime$.
Now, by denoting by $D$ the divisor
$D=\Theta+2\Theta_1+2\Gamma^{\prime\prime}$, we consider the exact
sequence
\begin{equation}\label{special}
0\to\mathcal O_{\mathcal Z^1}(-\mathcal Z^1_t+nH-D)\to\mathcal O_{\mathcal
Z^1}(n H-D)\to\mathcal O_{\mathcal Z^1_t}(nH-D)\to 0
\end{equation}
where $\mathcal Z^1_t$ is any fibre of $\mathcal Z^1$. By arguing as
in the previous lemma, we see that the image of the map
$r_0:\,H^0(\mathcal Z^1, \mathcal O_{\mathcal Z^1}(n H-D))\to
H^0(\mathcal Z^1_0,\mathcal O_{\mathcal Z^1_0}(n H-D))$
is equal to $h^0(\mathcal Z^1_,\mathcal O_{\mathcal Z^1_0}(n
H))-3=h^0(A\cup B,\mathcal O_{A\cup B}(n))-3$. Hence, if we denote
by $\mathcal T(n)_p\subset |\mathcal O_{\mathbb P^3}(nH)|$ the locus
of surfaces of $\mathbb P^3$, passing through $p$ and tangent to
$R=S^{d-1}\cap\pi$ at $p$, then all divisors in the linear series
$|\mathcal O_{\mathcal Z^1_0}(nH-D)|$ are cut out on $\mathcal
Z^1_0$ by surfaces parametrized by a divisor in $\mathcal T(n)_p$.
We denote this divisor by $\mathcal D_q$ where $q=\gamma\cap F$. How
to characterize $\mathcal D_q$? In order to answer this question we
contract $\Theta_1$, $\Theta$ and finally $\Gamma$, and we come back
to the family of surfaces $\mathcal X|_U\to U$. Now, the fibre $F$
of $ \mathcal E$ can be identified with a double cover of the fibre
over $p$ of the projectivized normal bundle to $R$ in $\mathbb P^3$.
In other words, every point $r$ of $F$ corresponds to a plane $H_r$
in  $\mathbb P^3$ containing the line $T_pR$ and there are exactly
two points of $F$ corresponding to the same plane. From the
hypothesis that the singular locus $\gamma$ of $\mathcal C$
intersects the fibre $F$ at $q$, we deduce that the image curve of
$\mathcal C_0$ in $\mathbb P^3$ is cut out on $S^{d-1}\cup \pi$ from
a surface of degree $n$ passing through $p$ and tangent at $p$ to
the hyperplane $H_q$ corresponding to $q\in F$. The locus of
surfaces with these properties is the divisor $D_q\subset\mathcal
T(n)$. By the generality of $q$ in $F$, we conclude that $T(d,n)$ is
an irreducible component of $\mathcal V_0$. To see that there are
not local sections of $\mathcal V_{nH,0,1}^{\mathcal S}$ passing
though the general point $[D]$ of $T(d,n)$ use the previous lemma.
\end{proof}
Now we want to see which is the limit curve on $S^{d-1}\cup \pi$ if
we consider a degeneration of a general one nodal curve on the
general fibre of the pencil $\mathcal F$, specializing to
$S^{d-1}\cup \pi$ in such a way that the node comes to a base point
of the pencil. Computations we are going to do are partially
contained in theorem 2.2 of \cite{chenversal}.
\begin{remark}\label{nodopuntobase}
Let $p_1 \in\mathbb P^3$ be the base point of our pencil $\mathcal
F$ corresponding to the exceptional divisor $E_1\subset B$. Let
$\gamma\subset\mathbb P^3$ be a general curve passing through $p_1$
and, in particular, transversally intersecting $S^{d-1}$ at $l(d-1)$
different points $p_1,\,q_2,\dots,\,q_{l(d-1)}$, where $l$ is the
degree of $\gamma$. If $\mathcal A\subset\mathbb A^3\times\mathbb
A^1$ is the total space of $\mathcal F$ and $g:\mathcal A\to\mathbb
A^3$ is the natural map, then, locally working at $(p_1,0)$, we may
suppose that $(p_1,0)=(0,0,0,0)$, that $\mathcal A$ has equation
$$xy-tz=0$$ and, finally, that $g^*\gamma$ has equations
$$\left\{
\begin{array}{l}
xy-tz=0\\
x-z=0\\
y-z=0.
\end{array}\right.$$
Notice that $(p_1,0)$ is a rational double point of $\mathcal A$ and
that the pull-back curve $g^*\gamma\subset\mathcal A$ has two
irreducible components passing through $(p_1,0)$,
\begin{eqnarray}
 \gamma_1:\left\{\begin{array}{l}
z=0\\
x=0\\
y=0
\end{array}\right. & \textrm{and} &
\gamma_2:\left\{
\begin{array}{l}
z=t\\
x=t\\
y=t,
\end{array}\right.
\end{eqnarray}
transversally intersecting at $(0,0,0,0)$ and where
$\gamma_1=g^*(p_1)$ is nothing else but the section of $\mathcal A$
corresponding to the inverse image of the point $p_1\in\mathbb P^3$
(which is a base point of $\mathcal F$). If we denote by
$\pi:\tilde{\mathcal A}\to\mathcal A$ the restriction to the proper
transform of $\mathcal A$ of the blowing-up of $\mathbb
A^3\times\mathbb A^1$ at $(p_1,0)$, then the special fibre of
$\tilde{\mathcal A}$ is given by $\tilde{\mathcal
A}_0=\tilde{A}\cup\Theta_1\cup \tilde{B}$, where $\tilde A$ is the
blowing up of $A=S^{d-1}$ at $p_1$, $\tilde B$ is the blowing up of
$\pi$ at $p_1$ and $\Theta_1$ is the quadric cut out on
$\tilde{\mathcal A}$ by the exceptional divisor. By identifying
$\Theta_1$ with the "projectivized" tangent cone of $\mathcal A$ at
$(0,0,0,0)$ we see that the curves $\pi^*\gamma_1$ and
$\pi^*\gamma_2$ intersect $\tilde{\mathcal A}_0$ at two different
points $q_1$ and $q_2$, corresponding to the tangent directions of
$\gamma_1$ and $\gamma_2$ at $(0,0,0,0)$ respectively. Moreover the
tangent direction to $\gamma_2$ at $(0,0,0,0)$ is determined by the
tangent direction to $\gamma$ at $p_1$ in $\mathbb P^3$. In
particular, if we set $E_1=\Theta_1\cap \tilde B$ and $E_1^\prime
=\Theta_1\cap\tilde A$, then $q_1$ and $q_2$ don't lie on $E_1\cup
E_1^\prime$ because $\gamma_1$ and $\gamma_2$ are not tangent to
$\pi$ or $S^{d-1}$ at $(0,0,0,0)$. Finally, notice that, the tangent
space to $\mathcal A\subset\mathbb P^3\times A^3\times \mathbb A^1$
at every point $(0,0,0,t)$, with $t\neq 0$, is given by the
hyperplane $H_z$ of equation $z=0$, and the curve $\gamma_1$ is
contained in $H_z$. Hence, the two lines $H_1$ and $H_1^\prime$ of
$\Theta_1$ with $H_1\in |E_1|$ and $H_1^\prime\in| E_1^\prime |$,
intersecting at $q_1$, are nothing else but the projectivization of
the intersection of $H_z$ with the tangent cone to $\mathcal A$ at
$(0,0,0,0)$.
\end{remark}
\begin{lemma}\label{lemmaW_E_i}
$W_{E_i}(d,n)$ is an irreducible component of $\mathcal V_0$ of
geometric multiplicity $1$, for every $i=1,\dots,\,d(d-1)$.
\end{lemma}
\begin{proof}
We will prove the lemma for $i=1$. We use the notation introduced in
Remark \ref{nodopuntobase}. We blow-up the ambient space at their
singular points and we contract all the exceptional components of
the proper transform of $\tilde{\mathcal A} $, different from
$\Theta_1$, on $\tilde B$. We denote by $\tilde {\mathcal
S}\to\mathbb A^1$ the family of surfaces obtained in this way.
$\tilde S$ is nothing else but the blowing-up of $\mathcal S$ along
$E_1$. We denote by $A^\prime\cup \Theta_1\cup B$ its special fibre.
Now, let $S\subset\tilde{\mathcal S}$ be a general divisor, linearly
equivalent to $nH-\alpha \Theta_1$, with nodal singularities along
$\gamma_2$. We ask which is the minimum $\alpha$ such that
$S|_{\Theta_1}$ is effective. Now, $ S_{|\Theta_1}\sim -\alpha
{\Theta_1}_{|\Theta_1}\sim\alpha E_1+\alpha E_1^\prime$. Thus the
minimum $\alpha$ such that $S|_{\Theta_1}$ is effective is
$\alpha=1$. Then take $S\sim nH-\Theta_1$. From the fact that $S$ is
nodal along $\gamma_2$, we have that $S|_{\Theta_1}$ has at least a
node at $\gamma_2\cap\Theta_1$. But $S|_{\Theta_1}\sim
E_1+E_1^\prime$ and hence $S|_{\Theta_1}=H_2\cup H^\prime_2$, where
$H_{2}\in |E_1|$ and $H_2^\prime\in |E_1^\prime|$ are the two lines
passing through $q_2$. In particular, we find that the curve $S\cap
(A^\prime\cup \Theta_1\cup B)$ is cut out from a surface of $\mathbb
P^3$ passing through $p_1$ and tangent at $p_1$ to a plane
$H_{q_2}$, determined by the point $q_2$. Moreover, if we blow-up
$\tilde{\mathcal S }$ along $\gamma_2$ and we denote by $\Gamma_2$
the exceptional divisor, by $\mathcal Z$ the resulting family of
surfaces and by $\Theta_1^\prime$ the proper transform of
$\Theta_1$, then, by arguing as in the last part of the proof of
lemma \ref{uno}, you can prove that the image of the restriction map
$$r_0:H^0(\mathcal O_{\mathcal Z}(nH-2\Gamma_2-\Theta_{1}^\prime))\to
H^0(\mathcal O_{\tilde A\cup \Theta_1^\prime\cup
B}(nH-2\Gamma_2-\Theta_1^\prime))$$ has dimension $ dim |\mathcal
O_{\tilde{\mathcal S_t}}(nH)|-3, $ and, in particular, $r_0$ is
surjective. So, the general curve on $\tilde{\mathcal S}_0=A\cup B$
cut out by a surface of $\mathbb P^3$ passing through $p_1$ and
tangent at $p_1$ to a plane $H_{q_2}$, is a limit of a general
one-nodal curve on $\mathcal S_t$ in the linear system $|nH|$.
Finally, by using the generality of the point $q_2$ in $\mathcal
E_1$ and by contracting $\Theta_1$ on $\tilde B$ we conclude our
proof.
\end{proof}
\begin{theorem}\label{conclusionenodo}
Let $\mathcal V_0$ be the special fibre of $\mathcal V_{nH,0,1}$.
Then, the irreducible components of $\mathcal V_0$ are
\begin{itemize}
\item $W_A(d,n)$, $W_{B}(d,n)$ and $W_{E_i}(d,n)$, for
$i=1,\dots,\,d(d-1)$, with geometric multiplicity $1$ and $T(d,n)$
with geometric multiplicity $2$, if $d,n\geq 2$;\\
\item $W_{E_1}(d,n)$ and $W_{E_2}(d,n)$ with geometric multiplicity $1$,
if $(d,n)=(2,1)$;\\
\item $W_A(d,n)$ and $W_{E_i}(d,n)$, for
$i=1,\dots,\,d(d-1)$, with geometric multiplicity $1$ and $T(d,n)$
with geometric multiplicity $2$, if $(d,n)=(d,1)$ and $d\geq 3$.
\end{itemize}
\end{theorem}
\begin{proof}
From what we proved before, we have only to show that $T(d,n)$,
$W_A(d,n)$, $W_{B}(d,n)$ and $W_{E_i}(d,n)$, for
$i=1,\dots,\,d(d-1)$, are the only irreducible components of the
special fibre $\mathcal V_0$ of $\mathcal V_{nH,0,1}$.

\textit{Case 1}. Let $V\subset\mathcal V_0$ be an irreducible
component whose general element $[D]$ corresponds to a curve
containing $E_i$. Thus $V\subset W_{E_i}(d,n)$. But
$dim(V)=dim(W_{E_i}(d,n))$, so $V= W_{E_i}(d,n)$.

\textit{Case 2}. Let $V$ be an irreducible component of $\mathcal
V_0$ whose general element $[D]$ corresponds to a curve $D=D_A\cup
D_{B}$ which is singular at a smooth point of $A\cup B$ and which
does not contain $E_i$, for every $i$. If $D_A$ is singular, thus
$V=W_A(d,n)$ and $D_{B}$ is smooth. If $D_{B}$ is singular, then
$V=W_{B}(d,n)$ and $D_A$ is smooth.

\textit{Case 3}. Assume that $V$ is an irreducible component of
$\mathcal V_0$ different from $W_A(d,n)$, $W_{B}(d,n)$ and
$W_{E_i}(d,n)$, for every $i$. Thus, from what we observed above, if
$D=D_A\cup D_{B}$ is the curve corresponding to the general element
$[D]$ of $V$, then $D_A$ and $D_{B}$ are both smooth curves. If
$U\subset\mathbb A^1$ is a sufficiently small analytic neighborhood
and $C\subset \mathcal V_{nH,0,1}|_U$ is a general $m$-multisection
passing through the point $[D]$, then the family of curves $\mathcal
D\to C$, naturally parametrized by $C$, has special fibre $\mathcal
D_0=mD$ and general fibre $\mathcal D_t=D_t^1\cup\dots\cup D_t^m$,
equal to the union of $m$ one-nodal irreducible curves in the linear
system $|\mathcal O_{\mathcal S_t}(n)|$. If we make a base change
\begin{displaymath}
\xymatrix{\mathcal X\ar@/_/[ddr]\ar[dr]\ar@/^/[drrr]_f& & &\\
& \mathcal \mathcal S^\prime \ar[d]\ar[r]_h &
\mathcal S\ar[d]\ar[r]& \mathbb P^3\\
& \mathbb A^1\ar[r]^{\nu_m}&\mathbb A^1 & }
\end{displaymath}
of order $m$ and we smooth the total space of the obtained family,
we get a family $\mathcal X\to\mathbb A^1$ which is a cover of order
$m$ of $\mathcal S\to\mathbb A^1$, totally ramified along its
special fibre. In particular, the special fibre of $\mathcal X$ is
$\mathcal X_0=A\cup \mathcal E_{1}\cup\dots\cup \mathcal E_{m-1}\cup
B$, where every $\mathcal E_{i}$ is a $\mathbb P^1$-bundle on an
irreducible curve $R_i$ isomorphic to $R$, for $i=1,\dots, m$, with
$A$ intersecting $\mathcal E_1$ along $R_1$,\,\dots, $\mathcal E_i$
intersecting $\mathcal E_{i-1}$ along $R_i\neq R_{i-1}$,\,\dots,
$\mathcal E_{m-1}$ intersecting $B$ along a section $R_m\simeq R$,
with $R_m\neq R_{m-1}$. The image into $\mathbb P^3$ of a section
$\gamma$ of $\mathcal X$, intersecting the special fibre at a point
$q\in \mathcal E_{i}$, is an analytic curve intersecting $A$ with
multiplicity $m-i$ at the point $p\in R$, image of $q$, and $B$ with
multiplicity $i$ at the same point. Now, we have $m$ irreducible
divisors $\mathcal D^i\subset\mathcal X$, $i=1,\dots,m$, in the
linear system $|nH|$, mapped to $\mathcal D\subset \mathcal S$, via
the morphism $\mathcal X\to\mathcal S$. For our purposes, it is
enough to consider only one of these divisors, say $\mathcal D^1$.
All fibres of $\mathcal D^1$ are now reduced. The general fibre of
$\mathcal D^1$ is an irreducible one-nodal curve, corresponding to a
general element of $\mathcal V_t$, while $\mathcal D^1$ cuts on
$\mathcal X_0$ a connected Cartier divisor which restricts to $D_A$
on $A$, to $D_{B}$ on $B$ and a union of fibres, counted with the
right multiplicity, on every $\mathcal E_{i}$. The singular locus of
$\mathcal D^1$ is a section of $\mathcal X$, which we denote by
$\gamma_1$.

Now, by the hypothesis that $D_B$ does not contain $E_i$, for every
$i$, we have that $\gamma_1$ does not intersect $\mathcal X_0$ at a
point on $E_i$ or on the fibre $F_{i,j}$ of $\mathcal E_j$, whose
image into $\mathcal S_0$ is the point $E_i\cap R$, for every $i$
and $j$. To see this, let $S\sim nH$ be a divisor whose singular
locus $\gamma$ intersects $E_i$ at a smooth point of $\mathcal X_0$
and let $\mathcal Y$ be the blowing-up of $\mathcal X$ along
$\gamma$, with exceptional divisor $\Gamma$. If $ S^\prime$ is the
proper transform of $S$ on $\mathcal Y$, then
$S^\prime\,E_i=-2\Gamma\,E_i=-2$, so $E_i\subset S^\prime$ and hence
$E_i\subset D_B\subset S\subset \mathcal X$. Similarly, if $\gamma$
intersects $\mathcal X_0$ at a smooth point of $F_{i,j}$, then, by
arguing as before, we see that $F_{i,j}\subset S$. Moreover, by
blowing-up along $F_{i,j}$, if $j<m-1$, we find that
$F_{i,j+1}\subset S$ and so on, until we get that $E_i\subset S$.

Finally, by the hypothesis that $D_A$ and $D_{B}$ are smooth and
from what we proved above, the curve $\gamma_1$ must intersect
$\mathcal X_0$ at a smooth point, say $q_1$, lying on a fibre
$F_{i}$ of $\mathcal E_{i}$, whose image point $p$ in $\mathbb P^3$
is not a base point of the pencil $\mathcal F$. Let $\mathcal Y$ be
the blowing-up of $\mathcal X$ along $\gamma_1$ and let us denote by
$\Gamma_1$ the new exceptional divisor. The proper transform, which
we still denote by $F_{i}$,  of $F_{i}$ on $ \mathcal Y_0$ is now a
$(-1)$-curve on the proper transform $\mathcal E_{i}^\prime$ of
$\mathcal E_{i}$. Now, if ${\mathcal D^1}^\prime$ is the proper
transform of $\mathcal D^1$, by the hypothesis that $\mathcal D^1$
is singular along $\gamma_1$, we have that ${\mathcal
D^1}^\prime\sim nH-2\Gamma_1$ and this implies that $F_{i}$ is
contained with multiplicity at least $2$ in the divisor  ${\mathcal
D^1}^\prime|_{\mathcal E_i}\sim 2F_i+cF_{\mathcal E_i}$, where
$c=n(d-1)-2$ and $F_{\mathcal E_i}$ is the linearly equivalence
class of the fibre of $\mathcal E_i$. By using that ${\mathcal
D^1}^\prime$ is a Cartier divisor, we find that $D_A={\mathcal
D^1}^\prime|_A$ and $D_B={\mathcal D^1}^\prime|_B$, intersect $R$
with multiplicity at least $2$ at the point $p$ of $R$ corresponding
to the fibre $F_{i}$. Since $D_A$ and $D_B$ are both smooth curves,
it follows that the curve $D_A\cup D_B\subset A\cup B$ is
parametrized by a point $[D_A\cup D_B]\in T(d,n)$. This proves the
statement.
\end{proof}
\begin{corollary}\label{nodosupuntodoppio}
Let $\mathcal V^\mathcal S_{nH,k,\delta}$ be the Universal
Severi-Enriques variety introduced in the first section. Let
$\mathcal V_0$ be the special fibre of $\mathcal V_{nH,k,\delta}$
and let $[D]\in\mathcal V_0$ be any point, corresponding to a curve
$D=D_A\cup D_B\subset A\cup B$. Assume that $D_A$ and $D_B$
intersect transversally $R=A\cap B$ at a point $p$, in such a way
that $D$ has a node at $p$. Let $U\subset\mathbb A^1$ be an analytic
neighborhood small enough of $0\in\mathbb A^1$ and let $C$ be a
general local $m$-multisection of $ \mathcal V_{nH,k,\delta}$
passing through $[D]$. Denote by $\mathcal D\to C$ the family of
curves naturally parametrized by $C$ and by $\mathcal D_t$ the
general fibre of $\mathcal D$, with irreducible components $\mathcal
D_t^1,\dots,\mathcal D^m_t$. Then, the point $p$ is not a limit of
any singular point of $\mathcal D^i_t$, for every $i=1,\dots,m$.
\end{corollary}
\begin{proof}
If we make a base change of order $m$ and we repeat the same
argument as in Case 3 of Lemma \ref{conclusionenodo}, we see that,
if $p$ is limit of a singular point of $D^i_t$, then $p$ is
contained with multiplicity at least $2$ in the divisor $R\cap
D_B=R\cap D_A$. Finally, notice that this is true for every point
$p\in R$, also if $p=E_i\cap R$, for some $i=1,\dots,d(d-1)$.
\end{proof}

\section{The case of one-cuspidal curves}\label{one cuspidal curves}
In this section we want to determine all irreducible components of
the special fibre $\mathcal V_0$ of $\mathcal V_{nH,k=1,\delta=0}$
with the respective geometric multiplicities (see definition
\ref{geometricmultiplicity}). We will assume $d,n\geq 2$, in such a
way that on the general surface $S_t\subset \mathbb P^3$ of the
pencil $\mathcal F$ there exist irreducible curves in the linear
system $|\mathcal O_{S_t}(n)|$ with a cusp at the general point of
$S_t$ and no further singularities.
In particular, under the hypothesis $d,n\geq 2$, we have that
$V_{nH_t,1,0}^{\mathcal S_t}$ is a non-empty, irreducible subvariety
of codimension $2$ of $|\mathcal O_{\mathcal S_t}(n)|$.
\begin{lemma}\label{molteplicita1cuspide}
Assume that $d\geq 2$ and $n\geq 3$. Then $$|\mathcal
O_{A}(n)|\times_{|\mathcal O_R(n)|}V^{B}_{nH,1,0}$$ is an
irreducible component of $\mathcal V_0$ of multiplicity $1$. If
$[D=D_A\cup D_{B}]$ is its general element, then $D_A$ and $D_{B}$
are irreducible, they intersect transversally $R$, $D_A$ is smooth
and $D_{B}$ has only one cusp as singularity. In particular, $D$
does not contain any exceptional divisor $E_i\subset B$, with
$i=1,\dots,d(d-1)$.
\end{lemma}
\begin{proof}
If $d\geq 2$ and $n=2$, then the variety $|\mathcal
O_{A}(n)|\times_{|\mathcal O_R(n)|}V^{B}_{nH,1,0}$ has dimension
smaller that $ dim(|\mathcal O_{A\cup B}(2)|)-2$ and, in particular
it cannot be an irreducible component of $\mathcal V_0$. On the
contrary, when $d\geq 2$ and $n\geq 3$ we have that $dim(|\mathcal
O_{A\cup B}(n)|)-2=dim(V_{nH_t,1,0}^{ S_t})$. We leave to the reader
to verify that the general element $[D=D_A\cup D_B]$ of $|\mathcal
O_{A}(n)|\times_{|\mathcal O_R(n)|}V^{B}_{nH,1,0}$ corresponds to a
curve as in the statement.
To prove that $[D]$ is a limit of the general element of $\mathcal
V_t$, consider in the linear series $|\mathcal O_{\mathcal S}(n)|$
the family of divisors singular along a section $\gamma$ of
$\mathcal S$ passing through a general point $p$ of $B$ and see that
it cuts out on $A\cup B$ a family of curves of codimension $4$ in
$|\mathcal O_{A\cup B}(n)|$, so it cuts out on $A\cup B$ all curves
with a cusp at $p\in B$ in the linear system $|nH|$. By the
generality of $p$, the statement is proved.
\end{proof}

\begin{lemma}
Assume that $d,n\geq 2$ and $(d,n)\neq (2,2)$. Then
$$V^{A}_{nH,1,0}\times_{|\mathcal O_R(n)|}|\mathcal O_{B}(n)|$$ is an
irreducible component of $\mathcal V_0$ of geometric multiplicity
$1$. If $[D=D_A\cup D_{B}]$ is its general element, the curve
$D_A\cup D_B$ does not contain any exceptional divisor $E_i\subset
B$,the curve $D_A$ and $D_{B}$ are irreducible, they intersect
transversally $R$, $D_B$ is smooth and $D_{A}$ has only one cusp as
singularity. Finally, $V^{A}_{nH,1,0}\times_{|\mathcal
O_R(n)|}|\mathcal O_{B}(n)|$ and $|\mathcal
O_{A}(n)|\times_{|\mathcal O_R(n)|}V^{B}_{nH,1,0}$ are the only
irreducible components of $\mathcal V_0$ of geometric multiplicity
$1$ and whose general element does not contain any $E_i$.
\end{lemma}
\begin{proof}
The proof is left to the reader.
\end{proof}
\begin{remark}
By the previous Lemmas and by Corollary \ref{nodosupuntodoppio}, the
other possible irreducible components of the special fibre $\mathcal
V_0$ of $\mathcal V_{nH,1,0}$ "are produced" by degenerations of the
general element $[D_t]$ of $\mathcal V_t$ such that as $D_t$ goes to
$A\cup B$, the cusp of $D_t$ specializes to the a point of the
singular locus $R$ of $A\cup B$.
\end{remark}
Let $F(d,n)\subset |\mathcal O_{A\cup B}(n)|$ be defined as the
Zariski closure of the quasi projective variety
\begin{eqnarray}
\{[D] & | &\,D_A\,\textrm{and}\, D_{B}\,\textrm{intersect}
\,R\,\textrm{at a general point }\,p\,\textrm{with
multiplicity}\,\,3\,\nonumber\\
& &\textrm{and they are smooth at}\,\,\,p\}.\nonumber
\end{eqnarray}
Notice that $F(d,n)$ parametrizes curves cut out on $A\cup B$ by
surfaces of degree $n$ intersecting $R$ with multiplicity $3$ at a
general point $p$ of $R$. If $d=n=2$, then $F(2,2)$ is cut out on
$A\cup B$ by quadrics containing $R$ and you can verify that
$dim(F(2,2))< dim(|\mathcal O_{A\cup B}(n)|)-2$. If $d,n\geq 2$ and
$(d,n)\neq (2,2)$, then $dim(F(d,n))=dim(|\mathcal O_{A\cup
B}(n)|)-2$ and moreover the general element $F(d,n)$ corresponds to
a curve $D=D_A\cup D_{B}$ such that $D_A$ and $D_{B}$ are smooth and
they intersect transversally $R$ at further $d(d-1)-3$ general
points, outside the multiplicity $3$ intersection point.
\begin{lemma}\label{flesso}
Let $(d,n)\geq (2,2)$, then $F(d,n)$ is a non-empty irreducible
component of the special fibre $\mathcal V_0$ of $\mathcal
V_{nH,1,0}$ if and only if $(d,n)\neq (2,2).$ The geometric
multiplicity of $F(d,n)$ is $2$.
\end{lemma}
\begin{proof}
\textit{Step 1}. Let $\mathcal X$ be the normalization of the double
cover of $\mathcal S\to\mathbb A^1$ ramified at the special fibre
\begin{displaymath}
\xymatrix{\mathcal X\ar@/_/[ddr]\ar[dr]\ar@/^/[drrr]_f& & &\\
& \mathcal \mathcal S^\prime \ar[d]\ar[r]_h &
\mathcal S\ar[d]\ar[r]& \mathbb P^3\\
& \mathbb A^1\ar[r]^{\nu_2}&\mathbb A^1 & }
\end{displaymath}
as in the proof of lemma \ref{tacnodo}. If $U$ is an analytic
neighborhood of $0\in \mathbb A^1$ small enough, let $\gamma$ be a
section of the family $\mathcal X|_U$ intersecting the special fibre
$\mathcal X_0=A\cup \mathcal E\cup B$ at a point $q$ lying on a
general fibre $F$ of $\mathcal E$. Now, let $\mathcal Y$ be the
blowing-up of $\mathcal X$ at $q$, with new exceptional divisor
$T\simeq \mathbb P^2$. The special fibre of $\mathcal Y$ now is
$A\cup \mathcal E^\prime\cup B\cup T$, where $\mathcal E^\prime$ is
the blowing-up of $\mathcal E$ at $q$, $F^2_{\mathcal E^\prime}=-1$,
$T$ intersects $\mathcal E^\prime$ along a curve which is a line on
$T$ and a $-1$-curve on $\mathcal E^\prime$ and it has not
intersections with $A$ and $B$. The proper transform
$\tilde{\gamma}$ of $\gamma$ now will intersect $T$ at a general
point $q_0$. Let us consider a general effective divisor $\mathcal
C\subset\mathcal Y|_U$ linearly equivalent to $nH-3T$ and with
cuspidal singularities along $\tilde\gamma$. Notice that such a
divisor $\mathcal C$ exists and its general fibre is an irreducible
curve with only a cusp as singularity at $q_t$. Now, if
$(d,n)=(2,2)$, then $(nH-3T)|_{\mathcal E^\prime}$ is not effective.
This implies that $\mathcal E^\prime$ is contained in the base locus
of $|nH-3T|$ and, if $C_A\cup C_B\subset A\cup B$ is the special
fibre of the image of $\mathcal C$ to $\mathcal S$, then $R= C_A\cap
C_B$. Finally, the point $[C_A\cup C_B]$ belongs to $F(2,2)$ but it
is not general in any irreducible component of $\mathcal V_0$.
Assume that $d,n\geq 2$ and $(d,n)\neq (2,2)$. Then
$(nH-3T)|_{\mathcal E^\prime}$ is an effective divisor containing
$F$ with multiplicity at least three. Moreover, $\mathcal C|_T$ will
be a cubic with at a least cusp at $q_0$.

\textit{Step 2}. Let $\mathcal Y^\prime$ be the blowing-up of
$\mathcal Y$ along $\tilde\gamma$ and denote by $\Gamma$ the new
exceptional divisor. The special fibre of $\mathcal Y^\prime$ is
$A\cup \mathcal E^\prime\cup B\cup T^\prime$, where $T^\prime$ is
the blowing-up of $T$ at $q_0$, with exceptional divisor
$E_0=T^\prime\cap \Gamma$. Moreover, notice that $\Gamma$ intersects
every fibre of $\mathcal Y^\prime$ along a $-1$-curve $E_t$. By the
hypothesis that $\mathcal C\subset\mathcal Y$ has cuspidal
singularities along $\tilde\gamma$, we have that the proper
transform $\mathcal C^\prime$ of $\mathcal C$ is linearly equivalent
to $nH-3T^\prime-2\Gamma$ and the general fibre ${\mathcal
C^\prime}_t$ of $\mathcal C^\prime$ intersects $E_t$ at only one
point $q_t^\prime$ and it smooth and tangent to $E_t$ at this smooth
point. The points $q_t^\prime$ determine a section $\psi$ of
$\mathcal Y^\prime$, contained in $\Gamma$, intersecting the special
fibre of $\mathcal Y^\prime$ at a general point $q_0^\prime\in
T^\prime\cap \Gamma$. If we blow-up $\mathcal Y^\prime$ along
$\psi$, we denote by $\mathcal Y^{\prime\prime}$ the obtained family
of surfaces and by $\Psi$ the new exceptional divisor, then the
proper transform $\mathcal C^{\prime\prime}$ of $\mathcal C^\prime$
will be linearly equivalent to $nH-3T^{\prime\prime}-2\Gamma-3\Psi$.
Now observe that $\mathcal C^{\prime\prime} F=-3T^{\prime\prime}
F=-3$, and hence $F\subset \mathcal C$. We set
$\alpha=mult_F(\mathcal C)=mult_F(\mathcal C^{\prime\prime})$.

\textit{Step 3}. Let $\mathcal Y^1$ be the blowing-up of $\mathcal
Y^{\prime\prime}$ along $F$. We denote by $\Theta_1$ the new
exceptional divisor. Now, the special fibre of $\mathcal Y^1$ is
given by $A_1+B_1+\Theta_1+\mathcal E^\prime+T_1$, where $A_1$,
$B_1$ and $T_1$ are the blow-ups of $A$, $B$ and $T^{\prime\prime}$
at $F\cap A$, $F\cap B$, and $F\cap T^{\prime\prime}$ respectively.
Then, by the triple point formula
$$
\Theta_1 (A_1+B_1+T_1+\mathcal E^\prime+\Theta_1)=0,$$ we find that
$F^2_{\Theta_1}=-2$ and hence $\Theta_1\simeq\mathbb F_2$. Moreover,
if $\mathcal C_1$ is the proper transform of $\mathcal
C^{\prime\prime}$, then, denoting by $F_{\Theta_1}$ the linear
equivalence class of the fibre of $\Theta_1$, we find that
$$
{\mathcal C_1}|_{\Theta_1}\sim  -3F_{\Theta_1}-\alpha
{\Theta_1}|_{\Theta_1} \sim -3F_{\Theta_1}+\alpha(3F_{\Theta_1}+F)
\sim (3\alpha-3)F_{\Theta_1}+\alpha F. $$ Now ${\mathcal
C_1}|_{\Theta_1}$ must be effective, so $\alpha\geq 1$. Moreover,
since $\mathcal C\subset \mathcal Y$ is general among divisors
linearly equivalent to $nH-3T$ and with cuspidal singularities along
$\tilde\gamma$, we may assume that $\alpha$ is the minimum integer
in order that ${\mathcal C_1}|_{\Theta_1}$ is effective, i.e.
$\alpha=mult_F(\mathcal C)=1$ and ${\mathcal C_1}|_{\Theta_1} =F$.
Again, notice that $F\subset \mathcal C_1$, because $F\, \mathcal
C_1=(-3T_1-\Theta_1)F=-3-(\Theta_1)^2 \mathcal E^\prime=-3+1=-2.$
Moreover, $\mathcal C_1$ must be smooth along $F$.

\textit{Step 4}. Let $\mathcal Y^2$ be the blowing-up along $F$ of
$\mathcal Y^1$. We denote by $\Theta_2$ the new exceptional divisor.
Now the special fibre $\mathcal Y^2_0$ of $\mathcal Y^2$ is
$$
\mathcal Y^2_0=A_2\cup \mathcal E^\prime\cup B_2\cup
T_2\cup\Theta_1\cup 2\Theta_2,$$ where $A_2$, $T_2$ and $B_2$ are
the blowing-up of $A_1$, $T_1$ and $B_1$ at $F\cap A_1$, $F\cap T_1$
and $F\cap B_1$ respectively. Now, by the triple point formula,
$F^2_{\Theta_2}=-1$ and $\Theta_2\simeq \mathbb F_1$. Moreover, if
$\mathcal C_2$ is the proper transform of $\mathcal C_1$, then
$\mathcal C_2|_{\Theta_2}\sim -2F_{\Theta_2}+\frac{1}{2}(3
F_{\Theta_2}+F+F_{\Theta_2}+F)$, where $F_{\Theta_2}$ is the linear
equivalence class of the fibre of $\Theta_2$. So, $\mathcal
C_2|_{\Theta_2}=F$.

\textit{Step 5.} Finally, let $\mathcal Y^3$ be the blowing-up of
$\mathcal Y^2$ along $F$ and let $\Theta_3\simeq \mathbb F_0$ be the
new exceptional divisor. The special fibre of $\mathcal Y^3$ now is
$$A_3+\mathcal E^\prime+T_3+B_3+\Theta_1+2
\Theta_2+3\Theta_3,$$ where $A_3$, $T_3$ and $B_3$ are the proper
transforms of $A_2$, $T_2$ and $B_2$. Moreover, if you denote by
$\mathcal C_3$ the proper transform of $\mathcal C_2$ then $\mathcal
C_3\sim nH-3T_3-2\Gamma-3\Psi-\Theta_1-2\Theta_2-3\Theta_3$ and the
linear system $|\mathcal C_3|_{\Theta_3}|=|F|$ is the ruling
determined by $F$ and it does not contain $F$ in its base locus.

Now, let $\mathcal D$ be the image of $\mathcal C\subset \mathcal
Y|_U$ into $\mathcal S|_{\nu_2(U)}$ and let $\mathcal D_0=2C_A\cup
2C_B$, where $ C_A=\mathcal C\cap A$ and $ C_{B}=\mathcal C\cap B$,
be the special fibre of $\mathcal D$. What we proved above shows
that \textit{the point $x=[ C_A\cup  C_{B}]\in F(d,n)\cap\mathcal
V_0$. In particular, the family $\mathcal D$ corresponds, into the
relative Hilbert Scheme $\mathcal H_n$, to an analytic local
bisection of $\mathcal V_{nH,1,0}^{\mathcal S}$ intersecting the
special fibre $\mathcal V_0$ at the point $x$ and the general fibre
$\mathcal V_t$ at two general points.}

\textit{Step 6}.\textit{ Now we want to prove that the point $x$ is
general in $F(d,n)$.} To this aim, recall that $F(d,n)\subset
|\mathcal O_{A\cup B}(n)|$ parametrizes curves cut out on $A\cup B$
by surfaces $S_n\subset \mathbb P^3$ intersecting $R$ with
multiplicity three at a general point. These surfaces are
parametrized by a codimension two subvariety $\mathcal F(n)\subset
|\mathcal O_{\mathbb P^3}(n)|$. Now, going back to the family of
surfaces $\mathcal X$ of Step 1, let $p$ be the point of $R$
corresponding to the fibre $F$ of $E$. As we already observed in
Lemma \ref{tacnodo}, $F$ can be identified with a double cover of
the projectivization of the fibre over $p$ of the normal bundle to
$R$ in $\mathbb P^3$. Equivalently, $F$ is a double cover of the
parameter space of planes of $\mathbb P^3$ containing the tangent
line $T_pR$ to $R$ at $p$. Let $H_q$ be the plane corresponding to
$q=\gamma\cap F$. Since the singular locus of the image of $\mathcal
C\subset \mathcal Y|_U$ into $\mathcal X|_U$ intersects $F$ at $q$,
then the special fibre $\mathcal C_0=\mathcal C\cap \mathcal Y_0$ of
$\mathcal C$ is cut out on $\mathcal Y_0$ by a surface $S_n$
intersecting $R$ at $p$ with multiplicity three and tangent to $H_p$
at $p$. \textit{In other words, if we denote by $F(d,n)_{p,H_q}$ the
codimension $4$ subvariety of $F(d,n)$, parametrizing curves cut out
by surfaces of $\mathbb P^3$ of degree $n$ intersecting with
multiplicity three $R$ at the point $p$ and tangent to the plane
$H_q$, then $x\in F(d,n)_{p,H_q}$.}

\textit{Step 7}. Now, let us denote by $D_{\Gamma,\Psi}$ the linear
equivalence class of the divisor
$3T_2+2\Gamma+3\Psi+\Theta_1+2\Theta_2+3\Theta_3\subset\mathcal Y^3$
and by $\mathcal Y^3_t$ the fibre of $\mathcal Y^3$ over $t\in
U\subset \mathbb A^1$. Then, by arguing as in the proof of lemma
\ref{uno}, we find that \textit{the dimension of the image}
$Im(r_0):=\mathcal W_{\Gamma,\Psi}$ \textit{of the map}
$$
r_0:H^0(\mathcal Y^3,\mathcal O_{\mathcal
Y^3}(nH-D_{\Gamma,\Psi}))\to H^0(\mathcal Y^3_0,\mathcal O_{\mathcal
Y^3_0}(nH-D_{\Gamma,\Psi}))
$$ \textit{is}

$$
h^0(\mathcal Y^3_t,\mathcal O_{\mathcal Y^3_t}(nH-D_{\Gamma,\Psi}))=
h^0(\mathcal Y^3_t,\mathcal O_{\mathcal Y^3_t}(nH-2\Gamma-3\Psi))=
h^0(\mathcal Y^3_t,\mathcal O_{\mathcal Y^3_t}(nH))-5.$$ \textit{We
want to prove that the family} $\mathcal V_{\Gamma,\Psi}\subset
|\mathcal O_{A\cup B}(nH)|$ \textit{of image divisors of divisors
in} $\mathcal W_{\Gamma,\Psi}$, \textit{with respect to the natural
morphism} $\mathcal Y^3\to\mathcal S$, \textit{has dimension}
$h^0(\mathcal Y_t,\mathcal O_{\mathcal Y_t}(nH))-5=h^0(\mathcal
S_t,\mathcal O_{\mathcal S_t}(nH))-5$. To this aim, notice that,
from what we have proved until now, at the Step 2, the restricted
linear system $|\mathcal C|_{T}|$ is the linear system $\mathcal
F_{\gamma,\psi}$ of cubics having a flex at the point $F\cap T$ and
a cusp, with cuspidal tangent line $R_\psi$ determined by the
section $\psi$ of $\Gamma$, at the point $\tilde\gamma\cap T$. So,
$dim(\mathcal F_{\gamma,\psi})\geq 1$. Actually, it is easy to show
that $\mathcal F_{\gamma,\psi}$ is a pencil whose all fibres are
irreducible and moreover, by using Proposition 2.1 of \cite{h}, one
can prove that, if $C_1$ and $C_2$ are two cubics of $\mathcal
F_{\gamma,\psi}$, then $C_1$ and $C_2$ intersect with multiplicity
exactly three at the point $F\cap T$. In particular, by using
notation of Step 5, the proper transforms $\tilde C_1$ and $\tilde
C_2$ of $C_1$ and $C_2$ to $T_3$ intersect the exceptional divisor
$\Theta_3\cap T_3$ at two different points $r_1$ and $r_2$. If $S_1$
and $S_2$ are two divisors in the family $\mathcal
W_{\Gamma,\Psi}\subset
|nH-(3T_3+2\Gamma+3\Psi+\Theta_1+2\Theta_2+3\Theta_3)|$ such that
$S_i|_{T_3}=\tilde C_i$, then the intersection points
$S_1|_{A_3}\cap \Theta_3$ and $S_2|_{A_3}\cap\Theta_3$ are different
and they are determined by $r_1$ and $r_2$. More precisely,
$S_i|_{A_3}\cap\Theta_3=F_{r_i}\cap A_3$, where $F_{r_i}$ is the
line of the ruling $|F|$ of $\Theta_3$, passing through $r_i$,
$i=1,\,2.$ We deduce, in particular, that there are not two divisors
in $\mathcal W_{\Gamma,\psi}$ restricting to the same divisor on $A$
and on $B$ and to two different cubics on $T_3$. \textit{So},
$\mathcal V_{\Gamma,\Psi}\subset F(d,n)\cap \mathcal V_0$
\textit{has dimension} $dim(\mathcal V_{\Gamma,\Psi})=dim(\mathcal
W_{\Gamma,\Psi})=dim(|\mathcal O_{\mathcal S_t}(nH)|)-5$.

Moreover, let us consider, at the Step 2, a general section $\psi_1$
of $\Gamma\subset\mathcal Y^\prime$, intersecting $T^\prime$ at a
point on $\Gamma\cap T^\prime$ different from $\psi\cap
T^\prime\in\Gamma\cap T^\prime$. \textit{We want to prove that}
$\mathcal V_{\gamma,\psi}$ \textit{and} $\mathcal V_{\gamma,\psi_1}$
\textit{are different subvarieties of} $F(d,n)\cap \mathcal
V_0\subset |\mathcal O_{A\cup B}(nH)|$. To this aim, let $\mathcal
Y^{\prime\prime}$ be the blowing-up of $\mathcal Y^\prime$ along
$\psi$ and $\psi_1$, let $\Psi$ and $\Psi_1$ the new exceptional
divisors, let us repeat all blow-ups of Steps 3, 4 and 5 and let us
use the same notation. Now, if $D$ and $D_1$ are two irreducible
cubics belonging respectively to the pencils $\mathcal
F_{\Gamma,\Psi}$ and $\mathcal F_{\Gamma,\Psi_1}$ on $T$, then $D$
and $D_1$ intersect with multiplicity $4$ at $\tilde\gamma\cap T$
and with multiplicity $m$, with $3\leq m\leq 5$ at the point $F\cap
T$. Moreover, for any cubic $C$ in the pencil $\mathcal
F_{\Gamma,\Psi}$, there exists only one cubic $C_1$ in the linear
system $\mathcal F_{\Gamma,\Psi_1}$, intersecting $C$ with
multiplicity at least $4$ in $F\cap T$. The proper transforms
$\tilde C$ and $\tilde{C_1}$ to $T_3$ of $C$ and $C_1$ will
intersect at a point $r$ of the exceptional divisor $\Theta_3\cap
T_3$ with multiplicity at most $2$. Let $S$ and $S_1$ be any two
divisors in $\mathcal Y^3$, belonging respectively to the linear
series $|nH-(3T_3+2\Gamma+3\Psi+\Theta_1+2\Theta_2+3\Theta_3)|$ and
$|nH-(3T_3+2\Gamma+3\Psi_1+\Theta_1+2\Theta_2+3\Theta_3)|,$ and such
that $S|_{T_3}=\tilde C$ and ${S_1}|_{T_3}=\tilde{C_1}$. We want to
prove that the curves $S\cap (A_3\cup B_3)$ and $S_1\cap (A_3\cup
B_3)$ can not be equal. Assume that $mult_r(\tilde C\cap
\tilde{C_1})=1$. Let $\mathcal Y_4$ be the blowing-up of $\mathcal
Y_3$ along the fibre $F_r$ of $\Theta_3$ passing through the point
$r$. (Notice that $F_r= S\cap \Theta_3=\Theta_3\cap S_1$). Let
$T_4$, $A_4$ and $B_4$ be the proper transforms of $T_3$, $A_3$ and
$B_3$ and let $\Theta_4$ be the new exceptional divisor. Now,
$\Theta_4$ is isomorphic to $\mathbb F_1$ and it is contained in the
special fibre $\mathcal Y_0^4$ of $\mathcal Y^4$ with multiplicity
$3$ and ${F_{r}}^2_{\Theta_4}=-1$. Moreover, if $S^\prime$ and
$S_1^\prime$ are the proper transform of $S$ and $S_1$ in $\mathcal
Y^4$ then
\begin{eqnarray*}
S^\prime|_{\Theta_4} & \sim &
(nH-3T_4-2\Gamma-3\Psi-\Theta_1-2\Theta_2-3\Theta_3-4\Theta_4)|_{\Theta_4}\\
& \sim &
-3T_4|_{\Theta_4}-3\Theta_3|_{\Theta_4}+\frac{4}{3}(3\Theta_3+T_4+A_4+B_4)|_{\Theta_4}\\
& \sim & H_{\Theta_4} \\
& \sim & S_1^\prime|_{\Theta_4},
\end{eqnarray*}
where $H_{\Theta_4}$ is the linear equivalence class of a line on
$\Theta_4$. The two lines $R=S^\prime\cap \Theta_4$ and
$R_1=S_1^\prime\cap\Theta_4$ intersect $T_4$ at two different point
by the hypothesis that $mult_r(\tilde C\cap \tilde{C_1})=1$ on
$T_3$. If $R$ and $R_1$ intersect $A_4$ at the same point, and hence
$S\cap A_3$ and $S_1\cap A_3$ are tangent at $F_r\cap A_3$, then $R$
and $R_1$ must intersect $B_4$ at two different points. In
particular, $S\cap B_3$ and $S^\prime\cap B_3$ intersect
transversally at $F_r\cap B_3$ and so they are different curves.
Assume now that $mult_r(\tilde C\cap\tilde{C_1})=2$ on $T_3$. Then,
when we blow-up along $F_r$, by using the same notation, the lines
$R=S^\prime\cap \Theta_4$ and $R_1=S_1^\prime\cap\Theta_4$ intersect
$T_4$ at the same point. If $R\neq R_1$ we find that $S\cap (A_3\cup
B_3)$ and $S_1\cap (A_3\cup B_3)$ intersect transversally at
$F_r\cap A_3$ and $F_r\cap B_3$ and, in particular, they are
different curves. If $R=R_1$ let $\mathcal Y_5$ be the blowing-up of
$\mathcal Y_4$ along $R$. The new exceptional divisor $\Theta_5$ is
still an $\mathbb F_1$ contained in the special fibre of $\mathcal
Y^5$ with multiplicity $6$. Moreover, the proper transforms
$S^{\prime\prime}$ and $S_1^{\prime\prime}$ of $S^\prime$ and
$S_1^\prime$ will intersect $\Theta_5$ along two lines $\tilde R$
and $\tilde R_1$ respectively. Now, $\tilde R$ and $\tilde R_1$ must
be different by the hypothesis that $mult_r(\tilde
C\cap\tilde{C_1})=2$ on $T_3$. This implies, in particular, that the
curves $S^{\prime\prime}\cap A_5\cup B_5$ and
$S_1^{\prime\prime}\cap A_5\cup B_5$ are different. \textit{This
proves that $\mathcal V_{\Gamma,\Psi}$ and $\mathcal
V_{\Gamma,\Psi_1}$ are two different subvarieties $F(d,n)\cap
\mathcal V_0$ of dimension $dim(|\mathcal O_{\mathcal
S_t}(nH)|)-5$.}

It follows that $F(d,n)\cap\mathcal V_0$ contains the codimension
$4$ subvariety $F(d,n)_{p,H_q}$ of $F(d,n)$. By using now the
generality of $q=\gamma\cap F$ in $F\subset \mathcal Y$ and the
generality of  the fibre $F$ on $\mathcal E$, we see that $F(d,n)$
is an irreducible component of $\mathcal V_0$.  Finally, the fact
that there are not local analytic sections of $\mathcal V_{nH,1,0}$
passing through the general element of $F(d,n)$ follows from lemma
\ref{molteplicita1cuspide}.
\end{proof}

\begin{lemma}\label{molteplicita2cuspide}
Assume $d,n\geq 2$. Let $V$ be an irreducible component of geometric
multiplicity $2$ of the special fibre $\mathcal V_0$ of $\mathcal
V^{\mathcal S}_{nH,1,0}$ whose general element $[D]$ corresponds to
a curve $D=D_A\cup D_B$, such that $D_B$ does not contain $E_i$, for
every $i\leq d(d-1)$.

\noindent Then $V=F(d,n)$. In particular, for $d=n=2$, there are not
irreducible components of $\mathcal V_0$ of geometric multiplicity
$2$.
\end{lemma}
\begin{proof}
Under the hypothesis $d,n\geq 2$, let $V$ be an irreducible
component of $\mathcal V_0$ as in the statement and let $[D=D_A\cup
D_B]$ its general element. Let
\begin{displaymath}
\xymatrix{ \mathcal X \ar[d]\ar[r]^h &
\mathcal S\ar[d]\\
\mathbb A^1\ar[r]^{\nu_2}&\mathbb A^1}
\end{displaymath}
be the smooth double cover of $\mathcal S$, totally ramified at the
special fibre, which we already described in Lemmas \ref{tacnodo}
and \ref{flesso}. Let $S\sim nH$ be a general divisor, such that
$S\cap \mathcal X_t$ is a general $1$-cuspidal curve on $\mathcal
X_t$ and such that $S\cap A=D_A$ and $S\cap B=D_B$. The singular
locus $\gamma$ of $S$ is a section of $\mathcal X$. By the
hypothesis that there are not local analytic sections of $\mathcal
V_{nH,1,0}$ passing through the general element of  $V$, we have
that $V\neq V^A_{nH,1,0}\times_{|\mathcal O_R(n)|}|\mathcal
O_{B}(n)|$ and $V\neq |\mathcal O_{A}(n)|\times_{|\mathcal
O_R(n)|}V^{B}_{nH,1,0}$, and so $\gamma\cap \mathcal X_0$ is not a
smooth point of $A\cup B$ lying on $A$ or $B\setminus
\cup_{i=1}^{d(d-1)} E_i$.

By the hypothesis that $D_B$ does not contain $E_i$, for every
$i$,by arguing as in Case 3 of Lemma \ref{conclusionenodo},  we have
that $\gamma$ does not intersect $\mathcal X_0$ at a point on $E_i$
or on the fibre $F_i$ of $\mathcal E$ passing through the point
$E_i\cap \mathcal E$.

Hence, $\gamma$ intersects $\mathcal X_0$ at a smooth point $q\in
F\subset \mathcal E$, where $F$ is any fibre of $\mathcal E$
different from $F_i$, for every $i$. As in the Step 1 of the
previous Lemma, let $\mathcal Y$ be the blowing-up of $\mathcal X$
at $q$ with exceptional divisor $T$ and let $\tilde\gamma$ be the
proper transform of $\gamma$. Now, the proper transform $S^\prime$
of $S$ is linearly equivalent to $nH-mT$, where $m$ is the
multiplicity of $S$ at $q$. Since $S^\prime|_T$ is a plane curve of
degree $m$ which must have at least a cusp at $\tilde\gamma\cap T$,
we have that $m\geq 2$.

If $d=n=2$ and $m\geq 3$ then $\mathcal E^\prime\subset nH-mT$. In
particular, $R_1\subset S|A=D_A$ and $R_2\subset S|_B=D_B$ and the
point $[D]$ is not general in any irreducible component of $\mathcal
V_0$.

Suppose that $m\geq 4$ and $(d,n)\neq (2,2)$. Then
$S^\prime|_{\mathcal E^\prime}$ contains $F$ with multiplicity
$m\geq 4$ and so $D_A\cup D_B$ is cut out on $A\cup B$ from a
surface of $\mathbb P^3$ intersecting $R=S^{d-1}\cap \pi$ with
multiplicity at least four at the point $p$ image of $F\subset
\mathcal E$. In particular, $[D_A\cup D_B]$ is general in a
subvariety $W\subset |\mathcal O_{A\cup B}(n)|$ of codimension at
least $4-1=3$ and it can not be general in any irreducible component
of $\mathcal V_0$.

If $m=3$ and $(d,n)\neq (2,2)$, then $V=F(d,n)$ by the previous
lemma.

Assume that $m=2$ and $d,n\geq 2$. Then, $S^\prime|_T=2R$, where $R$
is a line passing through the point $\tilde\gamma\cap T$. Now, by $F
S^\prime=-2$ we have that $S^\prime|_{\mathcal E^\prime}$ contains
$F$ with multiplicity $r\geq 2$. Since $S^\prime$ is a Cartier
divisor, $R$ is the line generated by $F\cap T$ and
$\tilde\gamma\cap T$ and $F$ is contained with multiplicity exactly
$2$ in the divisor $S^\prime|_{\mathcal E^\prime}$. Now, if
$\mathcal Y^{\prime\prime}$ is the blowing-up of $\mathcal Y^\prime$
along $F$, the proper transform $S^{\prime\prime}$ of $S^\prime$
restricts on the new exceptional divisor $\Theta_1\simeq \mathbb
F_2$ to an effective divisor linearly equivalent to
$$
-2F_{\Theta_1}+\alpha (3F_{\Theta_1}+F),
$$
where we may assume that $\alpha$ is the minimal integer in order
that $S^{\prime\prime}|_{\Theta_1}$ is effective and it intersects
with multiplicity two $T^\prime\cap \Theta_1$ at the point
$R^\prime\cap\Theta_1$, where $R^\prime$ and $T^\prime$ are the
proper transforms of $R$ and $T$. So $\alpha=2$ and $D_A$ and $D_B$
have a double point at the point $p\in R$ corresponding to the fibre
$F$. Also in this case $D_A\cup D_B$ is general in a subvariety
$W\subset |\mathcal O_{A\cup B}(n)|$ of codimension at least
$4-1=3$, and so $[D]=[D_A\cup D_B]$ can not be general in any
irreducible component of $\mathcal V_0$.
\end{proof}
 In order to describe the other irreducible components of
$\mathcal V_0\subset\mathcal V_{nH,1,0}$ we need to introduce some
notation. Let $S_A(d,n)=V^{A,R}_{nH,0,1}$ be the Zariski closure of
the locally closed set
$$
\{[D_A]|\,\,D_A\,\,\textrm{ has a node at a general point}\,\,
p\,\,\textrm{of}\,\,R\}\subset |\mathcal O_{A}(nH)|
$$
and let $S_{B}(d,n)=V^{B,R}_{nH,0,1}\subset |\mathcal O_{B}(n)|$ be
defined at the same way. Moreover, let $T_A(d,n)$ be the Zariski
closure of the locally closed set
$$
\{[D_A]|\,\,D_A\,\,\textrm{ is tangent to}\,\,R\,\,\textrm{at a
general point}\subset |\mathcal O_{A}(nH)|\}
$$
and let $T_{B}(d,n)\subset |\mathcal O_{B}(n)|$ be defined at the
same way.
\begin{lemma}\label{puntotriplo}
For every $d,n\geq 2$, we have that $S_A(d,n)\times_{|\mathcal
O_R(n)|}T_{B}(d,n)$ and $T_{A}(d,n)\times_{|\mathcal
O_R(n)|}S_{B}(d,n)$ are two irreducible components of the special
fibre $\mathcal V_0$ of $\mathcal V_{nH, 1,0}$ of geometric
multiplicity $3$.
\end{lemma}
Before proving the lemma notice that $S_A(d,n)\times_{|\mathcal
O_R(n)|}T_{B}(d,n)$ parametrizes curves cut out on $A\cup B$ by
surfaces of degree $n$ in $\mathbb P^3$ tangent to $S^{d-1}$ and
transverse to $\pi$ at the general point $p$ of $R=S^{d-1}\cap \pi$.
In particular, its general element $[D]$ corresponds to a curve
$D_A\cup D_B$, such that $D_A$ is a one nodal curve, with the node
at a general point $p$ of $R$, intersecting transversally $R$
outside $p$ at points different from $E_i\cap R$, for every $i$ and
$D_B$ is a smooth curve tangent to $R$ at $p$ and such that $D_B\cap
R=D_A\cap R$. Similarly for $T_{A}(d,n)\times_{|\mathcal
O_R(n)|}S_{B}(d,n)$.
\begin{proof}
We prove the lemma  for $S_A(d,n)\times_{|\mathcal
O_R(n)|}T_{B}(d,n)$. The other case is the same if you substitute
$A$ with $B$. Let $p\in R=A\cap B\subset \mathcal S_0$ be a point
different from $E_i\cap R$, for every $i$. By using the notation of
Theorem \ref{conclusionenodo}, we denote by $\mathcal Y$ the
desingularization of the triple cover of $\mathcal S$, totally
ramified along its special fibre $\mathcal Y_0=A\cup \mathcal
E_1\cup \mathcal E_2\cup B$ and by $F_i$ the fibre of $\mathcal
E_i$, with $i=1,\,2$, over the point $p=(0,0,0,0)$ of $\mathcal S$.
Let now $U\subset \mathbb A^1$ be an analytic neighborhood of $0$
and let $\gamma\subset \mathcal Y|_U$ be a local section passing
through a general point $q_0^1$ of $F_1$.

\textit{Step 1.} Let $\tilde{\mathcal Y}$ be the blowing-up of
$\mathcal Y$ at the point $q_0^1=\gamma\cap \mathcal Y_0$, with new
exceptional divisor $T$ and with special fibre $\tilde{\mathcal
Y}_0$. Let we denote by $\tilde{\gamma}$ the proper transform of
$\gamma$ to $\tilde{\mathcal Y}$ and let $$\tilde{\mathcal
C}\subset\tilde{\mathcal Y}|_U$$ be a general divisor such that
$\tilde{\mathcal C}\sim nH-2T$ and
the general fibre of $\tilde{\mathcal C}$ is irreducible with a cusp
as singularity at the intersection point $p_t^1=\tilde{\gamma}\cap
\tilde{\mathcal Y}_t$. Notice that, since we are assuming $d,n\geq
2$, such a divisor exists. Which kind of singularities may appear in
the special fibre of $\tilde{\mathcal C}$? First of all we observe
that $\tilde{\mathcal C}|_T=2L$, where $L$ is the line generated by
the points $\tilde{\gamma}\cap T$ and $F_1\cap T$ and so
$\tilde{\mathcal C}|_{\tilde{\mathcal Y}_0}$ contains $F_1$ with
multiplicity exactly $2$. Now, if we recontract $T$, the image
$\mathcal C$ of $\tilde{\mathcal C}$ in $\mathcal Y$ is a family of
curves with cuspidal singularities along $\gamma$ and such that
$\mathcal C|_{\mathcal Y^0}$ contains $F_1$ with multiplicity
exactly $2$. Let $\mathcal Y^1$ be the blowing-up of $\mathcal Y$
along $\gamma$. If $\Gamma$ is the new exceptional divisor, then
$\Gamma$ is a $\mathbb P^1$-bundle over $ \gamma$ intersecting every
fibre $\mathcal Y^1_t$ of $\mathcal Y^1$ along a curve $E_t$ which
is the exceptional divisor on $\mathcal Y_t$ and a fibre on
$\Gamma$. The special fibre of $\mathcal Y^1$ is $\mathcal
Y^1_0=A\cup \mathcal E_1^\prime\cup \mathcal E_2\cup B$, where
$\mathcal E_1^\prime$ is the blowing-up of $\mathcal E_1$ at
$\gamma\cap \mathcal E_1$, with exceptional divisor $E_0$. By the
hypothesis that $\mathcal C\subset\mathcal Y$ has cuspidal
singularities along $\gamma$, we have that the proper transform
$\mathcal C^\prime$ of $\mathcal C$ in $\mathcal Y^1$ is linearly
equivalent to $nH-2\Gamma$ and the general fibre $\mathcal
C^\prime_t$ of $\mathcal C^\prime$ is tangent at $E_t$ at a smooth
point $q_t^1$. Since $\mathcal C^\prime|_{\mathcal Y_1^0}$ contains
$F_1$ with multiplicity exactly $2$, the limit point $q_0^1$ of
$q_t^1$ will be the intersection point of $E_0$ with $F_1$. We
denote be $\phi$ the section described by points $q_t^1$. Notice
that $\Gamma\cap {\mathcal C}^\prime=2\phi$ and moreover $\mathcal
C^\prime_1 F_1=-2$.

\textit{Step 2.} Let us set $\alpha_1=mult_{F_1}\mathcal C^\prime$
and let $\mathcal Y^2$ be the blowing-up of $\mathcal Y^1$ along
$F_1$. The special fibre of $\mathcal Y^2$ is $\mathcal
Y^2_0=A^\prime\cup \mathcal E_1^\prime\cup \mathcal
E_2^\prime\cup\Theta_1\cup B$, where $\Theta_1$ is the new
exceptional divisor and $A^\prime$ and $\mathcal E_2^\prime$ are the
blowing-up of $A$ and $\mathcal E_2$ at $A\cap F_1$ and $\mathcal
E_2\cap F_1$ respectively. By the triple point formula, the new
exceptional divisor $\Theta_1$ is an $\mathbb F_1$ with exceptional
divisor $F_1=\Theta_1\cap \mathcal E_1^\prime$. Now, since $\phi$
intersects $\mathcal E_1^\prime$ transversally in $\mathcal Y^1$ at
$q_0^1\in F_1$, the proper transform $\phi^\prime$ of $\phi$ in
$\mathcal Y^2$, contained in the proper transform $\Gamma^\prime$ of
$\Gamma$, must intersect $\Theta_1$ at a point $q_1$ lying on the
fibre $F_{\Theta_1,q_0^1}$ of $\Theta_1$, corresponding to the point
$q_0^1$. If $\mathcal C^{\prime\prime}$ is the proper transform of
$\mathcal C^\prime$, then $\mathcal C^{\prime\prime}|_{\Theta_1}$
must be an effective divisor passing through $q_1$ and tangent to
$F_{\Theta_1,q_0^1}$ at $q_1$. Finally, by generality, $\alpha_1$ is
the minimum integer such that $\mathcal
C^{\prime\prime}|_{\Theta_1}$ verifies these properties. Now, if
$F_{\Theta_1}$ is the linearly equivalence class of the fibre of
$\Theta_1$, then
\begin{eqnarray*}
\mathcal C^{\prime\prime}|_{\Theta_1} & \sim & -2 F_{\Theta
_1}+\alpha_1(A^\prime+\mathcal E_1^\prime+\mathcal E_2^\prime)|{\Theta_1}\\
 & \sim & (2\alpha_1-2)F_{\Theta_1}+ \alpha_1F_1.
 \end{eqnarray*}
 From what we observed before, we may assume $\alpha_1=2$ and hence
 $\mathcal C^{\prime\prime}|_{\Theta_1}$ is a conic tangent to $F_{\Theta_1,q_0^1}$ at
 $q_1$ and verifying one more property. Indeed, since $\mathcal
 C^{\prime\prime}\sim nH-2\Gamma^\prime-2\Theta_1$ we have that
 $\mathcal C^{\prime\prime}F_2=-2\Theta_1F_2=-2$, and in particular
 $F_2\subset \mathcal C^{\prime\prime}$. More precisely,
 $\mathcal C^{\prime\prime}|_{\mathcal E_2^\prime}$
 contains $F_2$ with multiplicity $2$, because,
 as we observed at the previous Step, $\mathcal C^\prime|_{\mathcal E_1^\prime}
 \subset\mathcal Y^1$ contains $F_1$
with multiplicity $2$ and $\mathcal C^\prime|_{\mathcal Y^2_0}$ is a
Cartier divisor.

\textit{Step 3.} In order to understand the type of singularity of
$\mathcal
 C^{\prime\prime}|_{B}$ at the point $F_2\cap B$, let $\mathcal
 Y^3$ be the blowing-up of $\mathcal Y^2$ along $F_2$. The special fibre
 of $\mathcal Y^3$ is now
 $$\mathcal Y^3_0=A^\prime+\mathcal E_1^\prime+\Theta_1^\prime+
 \mathcal E_2^\prime+B^\prime+\Theta_2,$$
 where $\Theta_2$
 is the new exceptional divisor and $B^\prime$ and
$\Theta_1^\prime$ are the blowing-up of $B$ and $\Theta_1$ at $B\cap
F_2$ and $\Theta_1\cap F_2$ respectively. Again, by the triple point
formula, $\Theta_2$ is isomorphic to $\mathbb F_1$ with exceptional
divisor $F_2=\Theta_2\cap \mathcal E_2^\prime$. Now, if we set
$\alpha_2=mult_{F_2}\mathcal
 C^{\prime\prime}$, then $\alpha_2$ is the minimum integer in order that $\mathcal
 C^{\prime\prime\prime}|_{\Theta_2}$ is effective. By arguing as
 before, we find that
 \begin{equation*}
C^{\prime\prime\prime}|_{\Theta_2}\sim
(2\alpha_2-2)F_{\Theta_2}+\alpha_2F_2,
 \end{equation*}
where $F_{\Theta_2}$ is the linearly equivalence class of the fibre
of $\Theta_2$. Thus $\alpha_2=1$ and
$C^{\prime\prime\prime}|_{\Theta_2}$ is equal to $F_2$ and, again
using that $C^{\prime\prime\prime}|_{\mathcal Y^3_0}$ is a Cartier
divisor, $C^{\prime\prime\prime}|_{B}$ must contain the point
$F_2\cap B$ with multiplicity $1$. In particular, we have that, at
the Step 2, the divisor $\mathcal C^{\prime\prime}|_{\Theta_1}$ is a
smooth conic tangent to the fibre $\Theta_1\cap \mathcal E_2$ at the
point $F_2\cap\Theta_1$ and to the fibre $F_{\Theta_1,q_0^1}$ at the
point $q_1$. \textit{So, the divisor $\mathcal C\subset\mathcal
Y|_U$ cuts on $A$ a curve $C_A={\mathcal C}|_A$ with a node at
$F_1\cap A$ and on $B$ a curve $C_B={\mathcal C}|_B$ which is smooth
and simply tangent to $R_2$ at the point $F_2\cap B$. Moreover, if
$[C]\in |\mathcal O_{A\cup B}(n)|$ is the point corresponding to the
curve $C=C_A\cup C_B$, then $[C]\in \mathcal V_0\cap
S_A(d,n)\times_{|\mathcal O_R(n)|}T_{B}(d,n)$ and there exists a
local analytic trisection of $\mathcal V^{\mathcal S}_{nH,1,0}$,
passing through $[C]$ and intersecting the general fibre of
$\mathcal V^{\mathcal S}_{nH,1,0}$ at three general points.}

\textit{Step 4}. \textit{We want to prove that $[C]$ is a general
point of $S_A(d,n)\times_{|\mathcal O_R(n)|}T_{B}(d,n)$.}

\noindent Let $\mathcal Y^4$ be the blowing-up of $\mathcal Y^3$
along $ F_2$ and $\phi^\prime$, with new exceptional divisors
$\Theta_3$ and $\Phi$ respectively. The special fibre of $\mathcal
Y^4$ is
$$\mathcal Y^4_0=A^\prime+\mathcal E_1^\prime+\mathcal
E_2^\prime+\Theta_1^{\prime\prime}+\Theta_2+2\Theta_3+B^{\prime\prime},$$
where $\Theta_1^{\prime\prime}$ and $B^{\prime\prime}$ are the
proper transforms of $\Theta_1^\prime$ and $B^\prime$ respectively.
By the triple point formula, $\Theta_3\simeq \mathbb F_0$ and $F_2$
is a line on $\Theta_3$. Moreover, if $\tilde{\mathcal C}$ is the
proper transform of $\mathcal C^{\prime\prime\prime}$ in $\mathcal
Y^4$, then $\tilde{\mathcal C}\sim
nH-2\Gamma^\prime-3\Phi-2\Theta_1^{\prime\prime}-\Theta_2-2\Theta_3$
and $\tilde{\mathcal C}|_{\Theta_3}\sim F_2$. Now, denoting by
$\mathcal Y^4_t$ the general fibre of $\mathcal Y^4$ and by
$D_{q_1}$ the linear equivalence class of the divisor
$2\Gamma^\prime+3\Phi+2\Theta_1^{\prime\prime}+\Theta_2+2\Theta_3$,
by arguing as in Lemma \ref{uno}, we see that the image $\mathcal
W_{q_1}$ of the restriction map
$$
r_0:|\mathcal O_{\mathcal Y^4}(nH-D_{q_1}))|\to |\mathcal
O_{\mathcal Y^4_0}(nH-D_{q_1}))|
$$
has dimension $dim(|\mathcal O_{\mathcal Y^4_t}(nH)|)-5$. Moreover,
you can easily verify that also the restriction map
$$
F:\mathcal W_{q_1}\to |\mathcal
O_{{\Theta_1}^{\prime\prime}}(nH-D_{q_1}))|
$$
is surjective and, hence, the pencil $|\mathcal
O_{{\Theta_1}^{\prime\prime}}(nH-D_{q_1})|$ cuts on the curve
$\Theta_1^{\prime\prime}\cap A^\prime$ a $g^1_2$, which we denote by
$\mathcal L_{q_1}$, whose ramification points are $F_1\cap A^\prime$
and $R_{q_1}$, where $R_{q_1}$ is the intersection point of the
fibre $\Theta_1^{\prime\prime}\cap A^\prime$ and the proper
transform $L_{q_1}$ on $\Theta_1^{\prime\prime}$ of the line on
$\Theta_1$ generated by the points $q_1$ and $F_2\cap \Theta_1$.

\textit{Step 5} Now, notice that, under the natural map $\mathcal
Y_0^4\to\mathcal S_0$ contracting all exceptional components of
$\mathcal Y^4_0$, the variety $\mathcal W_{q_1}$ is mapped
injectively to a  codimension $5$ subvariety $\mathcal V_{q_1}$ of
$|\mathcal O_{S_0}(n)|=|\mathcal O_{A\cup B}(n)|$. This follows from
the fact that there are not two divisors $S_1$ and $S_2$ in the
linear system $|\mathcal O_{\mathcal Y^4}(nH-D_{q_1})|$ cutting out
the same divisor on $A^\prime$ and $B^{\prime\prime}$ and two
different conics to $\Theta_1^{\prime\prime}$, because the conic cut
out by $S_i$ on $\Theta_1^{\prime\prime}$ determine the intersection
points $S_i\cap A^\prime\cap \Theta_1^{\prime\prime}$. Finally,
notice that $\mathcal V_{q_1}\subset\mathcal V_0\cap
S_A(d,n)\times_{|\mathcal O_R(n)|}T_{B}(d,n)$.

Now, if $p_1$ is another general point of the fibre $F_{\Theta_1,
q_0^1}$, corresponding to the intersection of $\Theta_1$ with
another general section $\psi_1$ of $\Gamma^\prime$, and, if blow-up
along $\psi_1$ and we consider the varieties $\mathcal W_{p_1}$ and
$\mathcal V_{p_1}$, then $\mathcal W_{p_1}\neq \mathcal W_{q_1}$ and
$\mathcal V_{p_1}\neq \mathcal V_{q_1}$. Indeed, by the previous
Step, the linear series $\mathcal L_{q_1}$ and $\mathcal L_{p_1}$
are different because they have different ramification points. So,
for every point $q_1\in F_{\Theta_1, q_0^1}$, the variety $\mathcal
V_{q_1}$ is contained in a codimension $4$ subvariety $\mathcal
V_{\gamma}$ of $|\mathcal O_{A\cup B}(n)|$, contained in $\mathcal
V_0\cap S_A(d,n)\times_{|\mathcal O_R(n)|}T_{B}(d,n)$ and
parameterizing all curves on $A\cup B$ which are image of curves cut
out on $\mathcal X_0=A\cup \mathcal E_1\cup \mathcal E_2\cup B$ by
divisors in the linear system $nH$ with cuspidal singularity along
$\gamma$.

Moreover, if $\beta$ is another section of $\mathcal Y$,
intersecting $\mathcal Y_0$ at a smooth point $x_0^1\in F_1\subset
\mathcal E_1$, with $x_0^1\neq q_0^1$, and we construct the  related
variety $\mathcal V_{\beta}\subset |\mathcal O_{A\cup B}(nH)|$ then
$\mathcal V_{\beta}\neq \mathcal V_{\gamma}$. To see this, let us
came back to Step 1. Let $S_\beta$ and $S_\gamma$ be two divisor in
the linear system $|\mathcal O_{\mathcal Y}(n)|$ with cuspidal
singularities along $\beta$ and $\gamma$, respectively. Let
$\mathcal Y^1$ be the blow-up of $\mathcal Y$ along $\beta$ and
$\gamma$ and after along $F_1$. We denote by $S_\beta^\prime$ and
$S_\gamma^\prime$ the proper transforms of $S_\beta$ and $S_\gamma$
on $\mathcal Y^1$ and again by $\Theta_1$ the exceptional divisor of
the blowing-up along $F_1$. We know that
$S_\beta^\prime|_{\Theta_1}=C_\beta$ and
$S_\gamma^\prime|_{\Theta_1}=C_\gamma$ are two irreducible conics
tangent to $F_2\cap \Theta_1$. Moreover,$C_\beta$ and $C_\gamma$
cannot coincide because they are tangent to two different fibres of
$\Theta_1$, but they can intersect $A^\prime$ at the same points,
where $A^\prime$ is again the proper transform of $A$ on $\mathcal
Y^1$. Assume that $C_\beta\cap A^\prime=C_\gamma\cap
A^\prime=\{r_1,r_2\}$. Then $C_\beta$ and $C_\gamma$ intersect with
multiplicity exactly $2$ at the point $F_2\cap \Theta_1$, by the
Bezout theorem. Now, since when we blow-up twice $F_2$, the last
exceptional divisor is isomorphic to $\mathbb F_0$ and the pull-back
of $S$ and $S^\prime$ will restrict to a line in $|F_2|$ on
$\Theta_3$, it follows that the curves
$S_\beta^\prime|_{B}=S_\beta|_{B}$ and
$S_\gamma^\prime|_{B}=S_\gamma|_B$ intersect with multiplicity
exactly two at the point $F_2\cap B$. In particular,
$S_\beta|_{B}\neq S_\gamma|_{B}$.

We have proved that the locus $\mathcal W_{F_1}$ parametrizing
curves on $\mathcal Y_0$, cut out by divisors in the linear system
$|\mathcal O_{\mathcal Y}(n)|$ with cuspidal singularities along a
section of $\mathcal Y$, intersecting $\mathcal Y_0$ at a smooth
point of $F_1$, has dimension $dim(\mathcal W_{F_1})=dim(|\mathcal
O_{\mathcal Y_t}(n)|)-3$ and it is mapped one to one to the variety
$\mathcal V_p\subset \mathcal V_0\cap S_A(d,n)\times_{|\mathcal
O_R(n)|}T_{B}(d,n)$ parametrizing divisors in $|\mathcal O_{A\cup
B}(n)|$ cut out by surfaces tangent to $A$ and transverse to $B$ at
$p$. By the generality of $p$ on $R$, we find that
$S_A(d,n)\times_{|\mathcal O_R(n)|}T_{B}(d,n)$ is an irreducible
component of $\mathcal V_0$. The fact that  there are not local
analytic sections or bisections of $\mathcal V^{\mathcal
S}_{nH,1,0}$ passing through the general element of
$S_A(d,n)\times_{|\mathcal O_R(n)|}T_{B}(d,n)$ follows by the
previous lemmas of this section.
\end{proof}
\begin{lemma}
Let $V$ be an irreducible component of $\mathcal V_0$ of geometric
multiplicity $m\geq 3$, whose general element $[D]$ corresponds to a
curve $D\subset A\cup B$ not containing $E_i$ for every $i$. Then
$V$ is equal to $S_A(d,n)\times_{|\mathcal O_R(n)|}T_B(d,n)$ or
$T_A(d,n)\times_{|\mathcal O_R(n)|}S_B(d,n)$.
\end{lemma}
\begin{proof}
Let $V$ be an irreducible component of $\mathcal V_0$ as in the
statement. By the generality of $[D]$ in $V$, if $\sum_im_ip_i$,
with $p_i\neq p_j$ if $i\neq j$, is the divisor cut out by $D$ on
$R$, then we have that $m_i\leq 3$ for every $i$. Indeed, if
$W\subset |\mathcal O_{A\cup B}(n)|$ is an irreducible component of
the locally closed set
$$
\{[C]|C=C_A\cup C_B\subset  |\mathcal O_{A\cup B}(n)|\,\textrm{such
that}\,C\cap R=\sum_i m_i p_i,\,\textrm{with}\,m_i\geq
4\,\textrm{for some}\,i \},
$$
them $dim(W)\leq dim( |\mathcal O_{A\cup B}(n)|)-4+1=dim( |\mathcal
O_{A\cup B}(n)|)-3$ and so $W$ cannot be an irreducible component of
$\mathcal V_0$. Moreover, if $D\cap R=\sum_im_ip_i$ and for some $i$
we have that $m_i=3$ then $D_A$ and $D_B$ must be both smooth at
$p_i$, because otherwise $dim(V)\leq dim( |\mathcal O_{A\cup
B}(n)|)-4+1$. So, in this case we have that $[D]\in F(d,n)$ and
hence $V=F(d,n)$ and $m=2$.

Hence, we may assume that, if $D\cap R=\sum_im_ip_i$, then $m_i\leq
2$ for every $i$. Now, let
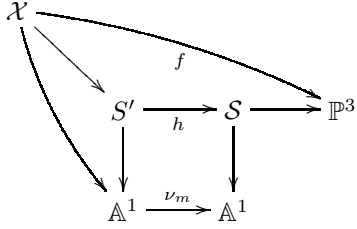
\begin{displaymath}
\xymatrix{\mathcal X\ar@/_/[ddr]\ar[dr]\ar@/^/[drrr]_f& & &\\
& \mathcal \mathcal S^\prime \ar[d]\ar[r]_h &
\mathcal S\ar[d]\ar[r]& \mathbb P^3\\
& \mathbb A^1\ar[r]^{\nu_m}&\mathbb A^1 & }
\end{displaymath}
be the covering of order $m$ of $\mathcal S$ totally ramified at the
special fibre, which we already introduced in the Lemma
\ref{conclusionenodo}. By using the same notation as in Lemma
\ref{conclusionenodo}, let $\mathcal X_0=A\cup\dots\cup\mathcal
E_{i}\cup\dots\cup B$ be the special fibre of $\mathcal X$ and let
$\mathcal D\subset\mathcal X$ be a divisor, linearly equivalent to
$nH$, such that $\mathcal D\cap A=D_A$, $\mathcal D\cap B=D_B$ and
$\mathcal D\cap \mathcal X_t$ is a general $1$-cuspidal curve on the
fibre in $|\mathcal O_{\mathcal X_t}(n)|$. By the hypothesis that
$D_B$ does not contain any exceptional divisor $E_l$, by using the
argument of Lemma \ref{molteplicita2cuspide}, we have that the
singular locus $\gamma$ of $\mathcal D$ intersects $\mathcal X_0$ at
a point $q$ of $\mathcal E_{i}$ lying on a fibre $F_{i}$, whose
image point $p$ in $\mathcal S_0$ is not $E_l\cap R$, for every
$l\leq d(d-1)$. Now let $\mathcal X^1$ be the blowing-up of
$\mathcal X$ along $\gamma$ with exceptional divisor $\Gamma$ and
special fibre $\mathcal X_0^1=A\cup\dots\cup\mathcal
E_{i}^\prime\cup\dots\cup B$, where $\mathcal E_{i}^\prime$ is the
blowing-up of $\mathcal E_{i}$ at $q$. We denote by $\mathcal D^1$
the proper transform of $\mathcal D$ in $\mathcal X^1$. By the
hypothesis that, if $D\cap R=\sum_im_ip_i$, then $m_i\leq 2$ for
every $i$, the divisor $\mathcal D^1|_{\mathcal E_{i}^\prime}$
contains $F_{i}$ with multiplicity exactly $2$. This implies that
$\mathcal D^1\cap \Gamma=2\psi$, where $\psi$ is a section of
$\Gamma$ intersecting $\mathcal E_{i}$ at the point
$q^1=F_{i}\cap\Gamma$. If $\alpha_{i}=mult_{F_{i}}\mathcal D_1$ and
if $\mathcal X^2$ is the blowing-up of $ \mathcal X^1$ along
$F_{i}$, with new exceptional divisor $\Theta_{i}\simeq \mathbb
F_1$, then
$$
\mathcal D^2|_{\Theta_{i}}\sim
-2F_{\Theta_{i}}+\alpha_{i}(2F_{\Theta_i}+F_{i}),
$$
where $\mathcal D^2$ is the proper transform of $\mathcal D^1$ in
$\mathcal X^2$, $F_{\Theta_{i}}$ is the linear equivalence class of
the fibre of $\Theta_{i}$ and $F_{i}$ is the $(-1)$-curve on
$\Theta_{i}$. Now, if $\Gamma^\prime$ and $\psi^\prime$ are the
proper transforms of $\Gamma$ and $\psi$ to $\mathcal X^2$ and
$F_{\Theta_{i},q^1}=\Gamma^\prime\cap \Theta_{i}$, then $\mathcal
D^2|_{\Theta_{i}}$ must be an effective divisor intersecting
$F_{\Theta_{i},q^1}$ with multiplicity $2$ at the point
$\psi^\prime\cap\Theta_{i}$. So $\alpha_{i}=2$ and $\mathcal
D^2|_{\Theta_{i}}$ is a conic.

\textit{Case 1}. If $i=1$, i.e. $\mathcal E_{i}=\mathcal
E_{1}\subset\mathcal X_0$ is the $\mathbb P^1$-bundle intersecting
$A$, then the special fibre of $\mathcal X^2$ is $\mathcal
X^2_0=A^\prime\cup\mathcal E_{1}^\prime\cup\Theta_{1}\cup\dots\cup
B$, where $A^\prime$ is the blowing-up of $A$ at the point
$F_{1}\cap A$ with exceptional divisor $A^\prime\cap\Theta_{1}$.
From what we have proved above, we have that $\mathcal D^2$ cuts on
$ A^\prime\cap\Theta_{1}$ a divisor of degree $2$ and so $D_A\subset
A$ has a double point at $p=F_{1}\cap A$. Now, if $D_B\subset B$ is
smooth at the point $p$, then the point $[D_A\cup D_B]$ belongs to
$S_A(d,n)\times_{|\mathcal O_R(n)|}T_B(d,n)$, so
$V=S_A(d,n)\times_{|\mathcal O_R(n)|}T_B(d,n)$ and, by the previous
lemma, the geometric multiplicity $m$ of $V$ is $3$. If $D_B$ is
singular at $p$ then the point $[D]$ cannot be general in any
irreducible component $V$ of $\mathcal V_0$.

\textit{Case 2}. If $i=m-1$ then, by substituting $A$ with $B$ in
the previous case, we find that $V=T_A(d,n)\times_{|\mathcal
O_R(n)|}S_B(d,n)$ and $m=3$.

\textit{Case 3}. Assume that $m\geq 4$ and $i\geq 2$. Also in this
case, we will prove that at least one of the curves $D_A$ and $D_B$
is singular at $p$. We denote by $F_{i-1}$ the fibre of $\mathcal
E_{i-1}$ passing through $F_{i}\cap \mathcal E_{i-1}$ and so on,  in
such a way that $F_{1}\cup\dots\cup F_{m-1}$ is a connected chain of
fibres, with $F_{\alpha}\subset \mathcal E_{\alpha}$, contained in
$\mathcal D|_{\mathcal X_0}$ with multiplicity $2$ and whose image
in $\mathcal S_0$ is the point $p\in R$. Now, the conic $\mathcal
D^2|_{\Theta_{i}}$ must intersect with multiplicity $2$ the fibre
$F_{\Theta_{i},q^1}$ at the point $\psi^\prime\cap\Theta_{i}$, the
fibre $\mathcal E_{i+1}\cap \Theta_{i}$ at the point $F_{i+1}\cap
\Theta_{i}$ and the fibre $\mathcal E_{i-1}\cap \Theta_{i}$ at the
point $F_{i-1}\cap \Theta_{i}$. So the points
$\psi^\prime\cap\Theta_{i}$, $F_{i+1}\cap \Theta_{i}$ and
$F_{i-1}\cap \Theta_{i}$ belong to a line $L_i$ and $\mathcal
D^2|_{\Theta_{i}}=2L_i$. Now, let $\mathcal X^3$ be the blowing-up
of $\mathcal X^2$ along $F_{i-1}$, $F_{i+1}$ and $\psi^\prime$, with
exceptional divisors $\Theta_{i-1}$, $\Theta_{i+1}$ and $\Psi$. We
denote by $\Theta_{i}^\prime$ the proper transform of $\Theta_{i}$
in $\mathcal X^3$. Now, ${L_i^2}_{\Theta_{i}^\prime}=-2$. Moreover,
by repeating always the same argument, we see that
$\Theta_{i-1}\simeq\mathbb F_1\simeq\Theta_{i+1}$ and, denoting by
$\mathcal D^3$ the proper transform of $\mathcal D^2$ in $\mathcal
X^3$, we have that $\mathcal D^3\cap\Theta_{i+1}$ and $\mathcal
D^3\cap\Theta_{i-1}$ are two conics intersecting respectively the
fibres $\Theta_{i}^\prime\cap\Theta_{i+1}$ and
$\Theta_{i}^\prime\cap\Theta_{i-1}$ at the points
$L_i\cap\Theta_{i+1}$ and $L_i\cap\Theta_{i-1}$ with multiplicity
$2$. In particular, denoting by $\Gamma^{\prime\prime}$ the proper
transform of $\Gamma^\prime$ in $\mathcal X^3$ and by $E$ the
$(-1)$-curve $E=\Psi\cap\Theta_i^\prime$, we have that
$$
\mathcal D^3\sim
nH-2\Gamma^{\prime\prime}-3\Psi-2\Theta_{i}^\prime-2\Theta_{i-1}-2\Theta_{i+1},$$
$$\mathcal D^3|_{\Theta_{i}^\prime}=2L_i+E\,\,\,\textrm{and}$$
$$\mathcal D^3\,
L_i=-3+2(\Theta_{i-1}+\Theta_{i+1})L_i-4=(2L_i+E)L_i=-3.$$ So, if
$\mathcal X^4$ is the blowing-up of $\mathcal X^3$ along $L_i$ and
if we denote by $\Theta_{L_i}$ the new exceptional divisor, then
$\Theta_{L_i}\simeq\mathbb F_0$ and ${L_i}^2_{\Theta_{L_i}}=0$.
Moreover, denoting by $\mathcal D^4$ the proper transform of
$\mathcal D^3$ and by $|F_{\Theta_{L_i}}|$ the ruling of
$\Theta_{L_i}$ different from $|L_i|$, we find that
$$\mathcal
D^4|_{\Theta_{L_i}}\sim-3F_{\Theta_{L_i}}+\alpha_{L_i}
(2F_{\Theta_{L_i}}+L_i),$$ where $\alpha_{L_i}=mult_{L_i}\mathcal
D^3$. Since $\mathcal D^4|_{\Theta_{L_i}}$ must be an effective
divisor, we find that $\alpha_{L_i}=2$ and $\mathcal
D^4|_{\Theta_{L_i}}\sim F_{\Theta_{L_i}}+2\,L_i$. Now we want to
prove that $\mathcal D^4|_{\Theta_{L_i}}$ contains the fibre $F_E\in
|F_{\Theta_{L_i}}|$ passing through the point $E\cap \Theta_{L_i}$.
To see this, let $\tilde{\mathcal X^4}$ be the blowing-up of
$\mathcal X^4$ along $E$. If $\Theta_E$ is the new exceptional
divisor, then $\Theta_E\simeq \mathbb F_0$ and, denoting by
$\tilde{\mathcal D^4}$ the pull-back of $\mathcal D^4$ to $\mathcal
X^4$, then
$$
\tilde{\mathcal D^4}\sim
nH-2\Gamma^{\prime\prime}-3\Psi-2\Theta_{i}^\prime-2\Theta_{i-1}-2\Theta_{i+1}
-4\Theta_{L_i}^\prime-5\Theta_E,$$
where $\Theta_{L_i}^\prime$ is the proper transform of
$\Theta_{L_i}$ in $\tilde{\mathcal X^4}$. In particular,
$$F_E\,\tilde{\mathcal
D^4}=-4(\Theta_{L_i}^\prime)\,F_E-5\Theta_E\,F_E=+4\Theta_E\,F_E-5\Theta_E\,F_E=-1.$$
By using that $(F_E)^2_{\Theta_{L_i}}=-1$, we have that
$F_E\subset\tilde{\mathcal X^4}$ and so, recontracting $\Theta_E$,
$F_E$ is contained in $\mathcal D^4|_{\Theta_{L_i}}$. Now, the fact
that $\mathcal D^3$ is singular along $L_i$ implies that the two
conics $\mathcal D^3|_{\Theta_{i-1}}=C_{i-1}$ and $\mathcal
D^3|_{\Theta_{i+1}}=C_{i+1}$ are singular respectively at the points
$L_i\cap\Theta_{i-1}$ and $L_i\cap\Theta_{i+1}$. If $i=2$ it follows
that $D_A$ is singular at $p$ and, similarly, if $i=m-2$ then $D_B$
is singular at $p$. Assume now that $2<i<m-2$. Then
$C_{i-1}=2L_{i-1}$ and $C_{i+1}=2L_{i+1}$, where
$L_{i-1}\subset\Theta_{i-1}\subset\mathcal X^3$ is the line joining
$F_{i-2}\cap\Theta_{i-1}$ and $L_i\cap\Theta_{i-1}$ and, similarly,
$L_{i+1}\subset\Theta_{i+1}$ is the line joining
$F_{i+2}\cap\Theta_{i+1}$ and $L_i\cap\Theta_{i+1}$. We will prove
now that $\mathcal D^{3}$ is singular along $L_{i-1}$. By using the
same argument, you can verify that $\mathcal D^3$ is singular also
along $L_{i+1}$. First we observe that, by the equality $\mathcal
D^3|_{\Theta_{i-1}}=C_{i-1}=2L_{i-1}$, it follows that, in $\mathcal
X^4$, the points $L_{i-1}\cap\Theta_{L_i}$ and
$L_{i+1}\cap\Theta_{L_i}$ stay on the same line $F_{L_i}\in |L_i|$
of $\Theta_{L_i}\simeq \mathbb F_0$ and $\mathcal
D^4|_{\Theta_{L_i}}=2F_{L_i}+F_E$. Now let $\mathcal X^5$ be the
blowing-up of $\mathcal X^4$ along $F_{i-2}$ with new exceptional
divisor $\Theta_{i-2}\simeq\mathbb F_1$ and let $\mathcal D^5$ be
the proper transform of $\mathcal D^4$ in $\mathcal X^5$. If we
denote by $\Theta_{i-1}^\prime$ the proper transform of
$\Theta_{i-1}\subset\mathcal X^3$ in $\mathcal X^5$, we have that
$(F_{i-2})^2_{\Theta_{i-2}}=-1$ and
$(L_{i-1})_{\Theta_{i-1}^\prime}^2=-1$. Moreover, by using that
$F_{i-2}\,\mathcal D^4=-2$, we find that the restricted linear
system $\mathcal D^5|_{\Theta_{i-2}}$ is a conic, intersecting with
multiplicity $2$ the fibre $\Theta_{i-1}^\prime\cap\Theta_{i-2}$ at
$L_{i-i}\cap\Theta_{i-2}$. In particular,
$$
\mathcal D^5\sim
nH-2\Gamma^{\prime\prime}-3\Psi-2\Theta_{i}^\prime-2\Theta_{i-1}^\prime
-2\Theta_{i+1}-4\Theta_{L_i}-2\Theta_{i-2}$$
and $L_{i-1}\,\mathcal
D^5=(-2\Theta_{i-2}-4\Theta_{L_i}+2\Theta_{i-2}+2\Theta_{L_i})L_{i-1}=-2$.
It is enough to blow-up along $L_{i-1}$ to see that $\mathcal D^5$
is singular along $L_{i-1}$ and so $\mathcal D^5|_{\Theta_{i-2}}$ is
a conic singular at $L_{i-1}\cap \Theta_{i-2}$. If $i-2=1$ this
implies that $D_A$ is singular at $p$. If $i-2>1$, then $\mathcal
D^5|_{\Theta_{i-2}}=2L_{i-2}$, where $L_{i-2}$ is the line joining
$L_{i-1}\cap\Theta_{i-2}$ and $F_{i-3}\cap\Theta_{i-2}$. Now we will
prove that $\mathcal D^5$ is singular along $L_{i-2}$. To this aim,
let $\mathcal X^6$ be the blowing-up of $\mathcal X^5$ along
$L_{i-1}$ and $F_{i-3}$ with exceptional divisors
$\Theta_{L_{i-1}}\simeq\mathbb F_1$ and $\Theta_{i-3}\simeq\mathbb
F_1$. Now, denoting by $\mathcal D^6$ and $\Theta_{i-2}^\prime$ the
proper transform of $\mathcal D^5$ and $\Theta_{i-2}$ on $\mathcal
X^6$, we have that,
\begin{itemize}
\item $\mathcal D^6|_{\Theta_{L_i}}$ is the double line $F_{L_{i-1}}$
joining $F_{L_i}\cap\Theta_{L_{i-1}}$ and
$L_{i-2}\cap\Theta_{L_{i-1}}$;\\
\item $\mathcal D^6|_{\Theta_{i-3}}$ is a conic intersecting with
multiplicity $2$ the fibre $\Theta_{i-2}^\prime\cap\Theta_{i-3}$ at
the point $L_{i-2}\cap \Theta_{i-3}$;\\
\item $(L_{i-2})^2_{\Theta_{i-2}}=-1$.
\end{itemize}
In particular we find that
$$
\mathcal D^6\sim
nH-2\Gamma^{\prime\prime}-3\Psi-2\Theta_{i}^\prime-2\Theta_{i-1}^\prime-
2\Theta_{i+1}-4\Theta_{L_i}^\prime-2\Theta_{i-2}-4\Theta_{L_{i-1}}-2\Theta_{i-3}$$
and $L_{i-2}\,\mathcal
D^6=L_{i-2}\,(-2\Theta_{i-2}-4\Theta_{L_{i-1}}-2\Theta_{i-3})=-2-4+4=-2.$
Now it is enough to blowing-up along $L_{i-2}$ to see that $\mathcal
D^6$ is singular along $L_{i-2}$. If $i-3=1$ this implies that $D_A$
is singular at $p$. If $i-3>1$, then $\mathcal
D^6|_{\Theta_{i-3}}=2L_{i-3}$ where $L_{i-3}$ is the line joining
$F_{i-4}\cap \Theta_{i-3}$ and $F_{i-2}\cap\Theta_{i-3}$. Moreover,
by the same argument as used to prove that $\mathcal D_5$ is
singular along $L_{i-2}$, you can verify that $\mathcal D^6$ is
singular along $L_{i-3}$. By repeating this argument at the end you
prove that $\mathcal D$ is singular along $F_1$ and so $D_A$ has a
double point at $p$. At the same way, you can prove that also $D_B$
has a double at $p$. Hence, if $2<i<m-1$ both curves $D_A$ and $D_B$
are singular at $p$ and $[D]$ cannot be general in any irreducible
component of $\mathcal V_0$.
\end{proof}
\begin{corollary}\label{tacnodosupuntodoppio}
Let $\mathcal V^\mathcal S_{nH,k,\delta}$ be the Universal
Severi-Enriques Variety introduced in the first section. Let
$\mathcal V_0$ be the special fibre of $\mathcal V^\mathcal
S_{nH,k,\delta}$ and let $[D]\in\mathcal V_0$ be any point,
corresponding to a curve $D=D_A\cup D_B\subset A\cup B$. Assume that
$D_A$ and $D_B$ are smooth and simply tangent to $R=A\cap B$ at a
point $p$, with $p\neq E_i\cap R$, for every $i$. Assume that $[D]$
is a point of geometric multiplicity $m$ of $\mathcal V_0$. Let
$U\subset\mathbb A^1$ be an analytic neighborhood small enough of
$0\in\mathbb A^1$ and let $\Delta$ be a general local
$m$-multisection of $ \mathcal V_{nH,k,\delta}$ passing through
$[D]$. Denote by $\mathcal D\to \Delta$ the family of curves
naturally parametrized by $\Delta$ and by $\mathcal D_t$ the general
fibre of $\mathcal D$, with irreducible components $\mathcal
D_t^1,\dots,\mathcal D^m_t$. Then, the point $p$ is not a limit of
any cusp of $\mathcal D^i_t$, for every $i=1,\dots,m$.
\end{corollary}
\begin{proof}
This follows by the proof of Lemma \ref{molteplicita2cuspide} if
$m=2$ and Cases 1, 2 and 3 of the previous Lemma if $m\geq 3$.
\end{proof}
Now, we denote by $T_{E_i}(d,n)\subset |\mathcal O_{B}(n)|$ the
closure, in the Zariski topology, of the locally closed set

$$
\{[D_B]|\,D_{B}=E_i\cup D^\prime_B,\,\textrm{where}\,\,
D^\prime_B\sim nH-E_i\,\, \textrm{is smooth and passes through}\,\,
E_i\cap R\}
$$
where $E_1,\dots,\,E_{d(d-1)}$ are the exceptional divisors of $B$.
\begin{lemma}\label{puntobasecuspide}
The variety
$$T_A(d,n)\times_{|\mathcal O_R(n)|}T_{E_i}(d,n)$$ is an irreducible
component of the special fibre $\mathcal V^0$ of $\mathcal V_{nH,
1,0}$ of geometric multiplicity $3$ and it is the only irreducible
component of $\mathcal V_0$ whose general element corresponds to a
curve containing $E_i$, for every $i=1,\dots,d(d-1)$.
\end{lemma}
\begin{proof}
Since the proof is the same for every $i$, we assume $i=1$.

Let $V$ be an irreducible component of $\mathcal V_0$, whose general
element $[D=D_A\cup D_B]$ corresponds to a curve $D_A\cup D_B$
containing $E_1$. First of all, we want to prove that
\begin{equation}\label{puntobasecuspideclaim1}
V\,\,\textit{has geometric multiplicity at least equal to}\,\,3.
\end{equation}

\textit{Case 1.} Assume that the geometric multiplicity of $V$ is
$1$. Then, if $\Delta$ is a general local analytic curve passing
through $[D]$ and $\mathcal D\to\Delta$ is the family of curves
naturally parametrized by $\Delta$, we have that the special fibre
of $\mathcal D$ is $\mathcal D_0=D_A\cup D_B$ and the singular locus
$\gamma$ of $\mathcal D$ intersects $E_1$ at a smooth point $q_1\in
E_1$ of $\mathcal S_0$. Now, let $\mathcal X$ be the blowing-up of
$\mathcal S$ along $E_1$ with exceptional divisor $\Theta_1 \simeq
\mathbb F_0$. The pull-back $\gamma^\prime$ of $\gamma$ to $\mathcal
Y$ now intersects $\Theta_1$ at a general point $q_1^\prime$.
Moreover, if $\mathcal D^\prime$ is the proper transform of
$\mathcal D$ in $\mathcal Y$, then $\mathcal D^\prime\sim
nH-\alpha_1\Theta_1$, where $\alpha_1=mult_{E_1}\mathcal D$ and
$\mathcal D^\prime $ has cuspidal singularities along
$\gamma^\prime$. Since $\mathcal D^\prime|_{\Theta_1}\sim
(nH-\alpha_1\Theta_1)|_{\Theta_1}\sim \alpha_1\,F_1 +\alpha_1\,F_2$,
where $|F_1|$ and $|F_2|$ are the two rulings of $\Theta_1$, we have
that the minimal $ \alpha_1$ such that $\mathcal
D^\prime|_{\Theta_1}$ as a cusp at $q_1^\prime$ is $\alpha_1=2$.
This implies that $D_A$ has a double point at $E_1\cap A$ and
$D_B=2E_1+D_B^\prime$, where $D_B^\prime\sim nH-2E_1$. So $D$ is cut
out on $A\cup B$ by a surface $S_n\subset \mathbb P^3$ singular at
the point $p_1$ corresponding to the exceptional divisor $E_1$. Thus
$[D]$ cannot be general in any irreducible component of $\mathcal
V_0$.

\textit{Case 2} Assume that $V$ has multiplicity two. Let $\mathcal
X$ be the normalization of the double covering of $\mathcal S$
totally ramified at its special fibre $\mathcal X_0=A\cup \mathcal
E\cup B$. The proper transform of $D$ on $\mathcal X_0$, which we
still denote by $D$ is the connected Cartier divisor which restricts
to $D_A$ on $A$, to $D_B$ on $B$ and to a union of fibres on
$\mathcal E$. Now, in $\mathcal X$, we can find divisors $\mathcal
D\sim nH$, such that $\mathcal D|_{\mathcal X_0}=\mathcal D_0=D$,
the general fibre $\mathcal D_t$ is a general one-cuspidal curve on
$\mathcal X_t$, and the singular locus of $\mathcal D$ is a section
$\gamma$ of $\mathcal X$ intersecting $\mathcal X_0$ at a smooth
point $q_1$ lying on $E_1$ or on the fibre $F_1$ of $\mathcal E$
intersecting $E_1$.

\textit{Case 2.1} Assume that $q_1\in E_1$. Let $\mathcal Y$ be the
blowing-up of $\mathcal X$ along $ E_1$, with new exceptional
divisor $\Theta_1\simeq\mathbb F_0$ and special fibre $A\cup
\mathcal E^\prime\cup\Theta_1\cup B$. Now, the pull-back of $F_1$ to
$\mathcal Y$, which we still denote by $F_1$, is a $(-1)$-curve
intersecting transversally $\Theta_1$, whereas the pull-back of
$\gamma$ is a curve $\gamma^\prime$ intersecting $\Theta_1$ at a
general point. Moreover, $F_1\,\mathcal
D^\prime=-\alpha\,\Theta_1\,F_1=-\alpha$, where
$\alpha=mult_{F_1}\mathcal D$ and $\mathcal D^\prime$ is the proper
transform of $\mathcal D$ on $\mathcal Y$. So $F_1\subset\mathcal
D^\prime$ and, denoting by $|H_1|$ and $|H_2|$ are the two rulings
of $\Theta_1$, we have that $\mathcal
D^\prime|_{\Theta_1}\sim\alpha\,(H_1+H_2)$ is an effective divisor
with a cusp at $\gamma^\prime\cap \Theta_1$ and intersecting with
multiplicity two $\Theta_1\cap \mathcal E^\prime$ at the point
$F_1\cap \Theta_1$. The minimal $\alpha$ such that these two
conditions are verified is $\alpha=2$ and so $\mathcal
D^\prime|_{B}$ intersects $E_1$ at two points and, by contracting
$\Theta_1$ on $E_1$, we find that $D_B=2E_1+D_B^\prime$, where
$D_B^\prime\sim nH-2E_1$. Moreover, by blowing-up twice $\mathcal Y$
along $F_1$, we see that $\mathcal D^\prime|_A=D_A$ is smooth and
tangent to $\mathcal E^\prime\cap A$ at $F_1\cap A$. So $D_A\cup
D_B$ is cut out by a surface in $\mathbb P^3$ tangent to $B$ at the
base point $p_1$ corresponding to $E_1$. It follows that $[D]$ is
general in a family of codimension at least $3$ in $|\mathcal
O_{A\cup B}(n)|$ and it cannot be general in any irreducible
component of $\mathcal V_0$.

\textit{Case 2.2} Assume that $q_1\in F_1$. Then, let $\mathcal Y$
be the blowing-up of $\mathcal X$ along $\gamma$ with exceptional
divisor $\Gamma$ and special fibre $A\cup \mathcal E^\prime\cup B$.
Now, $\mathcal D^\prime\cap F_1=-2\Gamma\,F_1=-2$. In particular, if
$\mathcal D^\prime$ is the proper transform of $\mathcal D$, then
$F_1$ is contained in $D^\prime|_{\mathcal E^\prime}$ with
multiplicity $m\geq 2$. If $m\geq 3$ then $D_A$ intersects $\mathcal
E^\prime$ with multiplicity $m\geq 3$ at $F_1\cap A$ and so $[D]$
cannot be general in any irreducible component of $\mathcal V_0$.
Hence $m=2$ and, if we blow-up $\mathcal Y$ along $F_1$ and we
denote by $\Theta_1$ the new exceptional divisor and by $\mathcal
D^{\prime\prime}$ the proper transform of $\mathcal D^\prime$, we
find that $\mathcal D^{\prime\prime}|_{\Theta_1}$ is a conic,
tangent to the fibre $\Gamma^\prime\cap\Theta_1$, where
$\Gamma^\prime$ is the pull-back of $\Gamma$. Moreover, since
$\mathcal D^{\prime\prime}\,E_1=-2$, we have that
$E_1\subset\mathcal D^{\prime\prime}$ and the conic $\mathcal
D^{\prime\prime}|_{\Theta_1}$ passes through $E_1\cap \Theta_1$. In
particular, recontracting $\Theta_1$, we find that $D_A$ has a
double point at $F_1\cap A$ whereas $D_B=E_1\cup D_B^\prime$, where
$D_B^\prime\sim nH-E_1$. So $D_A\cup D_B$ is cut out on $A\cup B$ by
a surface $S_n\subset\mathbb P^3$ tangent to $A$ at the point $p_1$
corresponding to the exceptional divisor $E_1$, and it cannot be
general in any irreducible component of $\mathcal V_0$.

\textit{Claim \eqref{puntobasecuspideclaim1} has been proved}. Now
we observe that, if $V$ is an irreducible component of $\mathcal
V_0$, whose general element $[D]$ corresponds to a curve $D=D_A\cup
D_B$ containing $E_1$, then the singularity of $D$ at the point
$E_1\cap R$ must impose at most two conditions to the linear system
$|\mathcal O_{A\cup B}(n)|$. So, we have one of the following two
cases:
\begin{enumerate}
\item $D$ has a node at $E_1\cap R$, in particular
$D_A$ meets transversally $R$ at $E_1\cap R$ and $D_B=E_1\cup
D_B^\prime$, where $D_B^\prime$ does not contain $E_1\cap
R$;\label{transverseintersection}
\item $D_A$ is tangent to $R$ at $E_1\cap R$ and $D_B=D_B^\prime\cup
E_1$, where $D_B^\prime\sim nH-E_1$ passes through $E_1\cap R$. In
particular, $V= T_A(d,n)\times_{|\mathcal O_R(n)|}T_{E_i}(d,n)$.
\end{enumerate}
\textit{We want to prove that the case
\eqref{transverseintersection} cannot occur.} To this aim let $m$ be
the geometric multiplicity of $V$ and let $\mathcal Y$ be the finite
covering of degree $m$ of $\mathcal S$ which we already introduced
in Lemma \ref{conclusionenodo}. Let $\mathcal Y_0=A\cup\mathcal
E_1\cup\dots\cup\mathcal E_{m-1}\cup B$ be its special fibre. We
denote by $F_1\cup\dots\cup F_m$ the connected chain of fibres such
that $F_i\subset \mathcal E_i$ and $F_{m-1}\cap B=E_1\cap R$. Let
$\mathcal D\sim nH$ be a general divisor in $\mathcal Y$ cutting out
$D_A$ on $A$ and $D_B$ on $B$ with cuspidal singularity along a
section $\gamma$ of $\mathcal Y$. We already know that $\gamma$ must
intersect $\mathcal Y_0$ at a smooth point $q$ lying on $E_1$ or on
$F_i$, for some $i$. If $q\in E_1$ then, arguing as in Case 2.1, we
see that $D_B$ contains $E_1$ with multiplicity $2$. If $q\in F_i$,
then $\mathcal D|_{\mathcal Y_0}$ contains every $F_i$ with
multiplicity $r\geq 2$ and so $D_A\cap R=D_B\cap R$ contains
$E_1\cap R$ with multiplicity $r\geq 2$. \textit{This prove that
case \eqref{transverseintersection} cannot occur and}
$V=T_A(d,n)\times_{|\mathcal O_R(n)|}T_{E_i}(d,n).$

\textit{Now, we will show that, actually,
\begin{equation}\label{puntobasecuspideclaim2}
T_A(d,n)\times_{|\mathcal O_R(n)|}T_{E_i}(d,n)\,\textrm{is an
irreducible component of}\,\, \mathcal V_0\,\,\textrm{of
multiplicity}\, \,3.
\end{equation}}
Assume $m=3$ and let $\mathcal Y$, $\mathcal D$ and
$F_i\subset\mathcal E_i$ be as before. We denote by $\gamma$ a
section of $\mathcal Y$ intersecting $\mathcal Y_0=A\cup\mathcal
E_1\cup\mathcal E_2\cup B$ at a smooth point $q\in F_2$.

\textit{Step 1.} Let $S\subset\mathcal Y$ be a general divisor in
$|\mathcal O_{\mathcal Y}(nH)|$ with cuspidal singularities along
$\gamma$ and such that $S$ contains the fibres $F_1$ and $F_2$ with
multiplicity exactly $2$. Let $\mathcal Y^1$ be the blowing up of
$\mathcal Y$ along $\gamma$ with new exceptional divisor $\Gamma$.
Denote by $S^1$ the proper transform of $S$ in $\mathcal Y^1$ and by
$\mathcal Y^1_0=A\cup \mathcal E_1\cup\mathcal E_2^\prime\cup B$ the
special fibre of $\mathcal Y^1$, where $\mathcal E_2^\prime$ is the
blowing up of $\mathcal E_2$ at $q$. Then $S^1\sim nH-2\Gamma$,
$S^1\,F_2=-2$ and, by the hypothesis that $S$ contains $F_2$ with
multiplicity exactly $2$, $S^1$ will be tangent to $\Gamma$ along a
smooth section $\psi$ of $\mathcal Y^1$, intersecting $\mathcal
Y^1_0$ at the point $F_2\cap\Gamma$, which we still denote by $q$.

\textit{Step 2.} Let $\mathcal Y^2$ be the blowing-up of $\mathcal
Y^1$ along $F_2$, with new exceptional divisor $\Theta_2$. Denote by
$\mathcal Y^2=A\cup \mathcal E_1^\prime\cup \mathcal E_2^\prime\cup
\Theta_2\cup B$ the special fibre of $\mathcal Y^2$, where $\mathcal
E_1^\prime$ is the blowing-up of $\mathcal E_1$ at $F_2\cap \mathcal
E_2$. By using that ${F_2}^2_{\mathcal E_2^\prime}=-1$, we have that
$\Theta_2\simeq\mathbb F_1$ and ${F_2}^2_{\Theta_2}=-1$. Moreover,
denoting by $S^2$ the proper transform of $S^1$ in $\mathcal Y^2$,
we have that
$$S^2|_{\Theta_2}\sim -2F_{\Theta_2}+\alpha (2F_{\Theta_2}+F_2),$$
where $F_{\Theta_2}$ is the linearly equivalence class of the fibre
of $\Theta_2$. Now, denoting by $\psi^\prime$ the proper transform
in $\mathcal Y^2$ of $\psi$ and by $q^\prime$ the intersection point
of $\psi^\prime$ with $\Theta^2$, we have that $S^2|_{\Theta_2}$ is
an effective divisor
\begin{enumerate}\label{primedue}
\item tangent to the
fibre $F_{q}\in |F_{\Theta_2}|$, passing through $q$, at the point $q^\prime$;\label{1}\\
\item intersecting the fibre $\mathcal E_{1}^\prime\cap\Theta_2$ with
multiplicity $2$ at the point $F_1\cap\Theta_2$.\label{2}
\end{enumerate}
So, $\alpha=2$, $S^2|_{\Theta_2}$ is a conic verifying \eqref{1} and
\eqref{2} and $S^2\sim nH-2\Gamma^\prime-2\Theta_2$, where
$\Gamma^\prime$ is the proper transform of $\Gamma$ in $\mathcal
Y^2$. Moreover, the pull-back of $E_1$ to $\mathcal Y^2$, which we
still denote by $E_1$, is a $(-2)$-curve transversally intersecting
$\Theta_2\cap B^\prime$ at a point $q_1$. So,
$E_1\,S^2=-2\Theta_2\,E_1=-2$, $E_1\subset S^2$ and the conic
$S^2|_{\Theta_2}$ passes through $q_1$. Now, if $\mathcal Y^3$ is
the blowing-up of $\mathcal Y^2$ along $E_1$, with new exceptional
divisor $\Theta_{E_1}$, then $(E_1)^2_{\Theta_{E_1}}=1$ and
$\Theta_{E_1}\simeq \mathbb F_1$. Moreover, the proper transform
$\Theta_2^\prime$ of $\Theta_2$ in $\mathcal Y^3$ is the blowing-up
of $\Theta_2$ at $q_1$ with new exceptional divisor
$\Theta_{E_1}\cap\Theta_2^\prime$. Finally, denoting by $S^3$ the
proper transform of $S^2$ in $\mathcal Y^3$, we have that
$$
S^3|_{\Theta_{E_1}}\sim -2F_{\Theta_{E_1}}+E_1+F_{\Theta_{E_1}}\sim
E_1-F_{\Theta_{E_1}},$$ where $F_{\Theta_{E_1}}$ is the linearly
equivalence class of a fibre of $\Theta_{E_1}$ and, as we already
observed, $E_1$ is a line. Hence,
$$
S^3|_{\Theta_{E_1}}\,\,\textrm{is the exceptional divisor
of}\,\,\Theta_{E_1}.$$ This implies that
\begin{itemize}
\item[(3)] $S^2|_{\Theta_2}$ is a conic verifying properties \eqref{1} and
\eqref{2} and tangent at $q_1$ to a fixed line $r_1\subset\Theta_2$
which does not depend on the sections $\gamma$ or $\psi$ and on the
divisor $S$.
\end{itemize}
Now, the family of conics on $\Theta_2$ tangent to $\mathcal
E^\prime_1\cap \Theta_2$ at $F_1\cap\Theta_2$ and to $r_1$ at $q_1$
is a pencil $\mathcal G$ cutting out on the fibre $F_{q}$ a $g_2^1$
with ramification points $x_{q}^1$ and $x_{q}^2$. It follows that
\begin{equation}\label{divisoresutheta2}
q^\prime=x_{q}^i\,\textrm{for
some}\,\,i=1,\,2\,\,\textrm{and}\,\,S^2|_{\Theta_2}=C_q^i,
\end{equation}
 where $C_q^i$ is the only conic of the pencil $\mathcal G$ tangent to
$F_q$ at $x_q^i$.  Finally, if
$\mathcal Y^4$ is the blowing-up of $\mathcal Y^3$ once along $\psi$
and twice along $F_1$, with new exceptional divisors $\Psi$,
$\Theta_1\simeq\mathbb F_1$ and $T_1\simeq\mathbb F_0$, then it is easy to see that
$$S^4\in |\mathcal O_{\mathcal
Y^4}(nH-2\Gamma^\prime-3\Psi-2\Theta_2^{\prime\prime}-\Theta_{E_1}-\Theta_1-2T_1)|,$$
where $S^4$ and $\Theta_2^{\prime\prime}$ are the proper transforms
of $S^3$ and $\Theta_2^\prime$ to $\mathcal Y^4$.  In particular,
$[D_A\cup D_B]\in T_A(d,n)\times_{|\mathcal O_R(n)|}T_{E_i}(d,n)$
and, more precisely, $S$ cuts on $A$ a curve $D_A$ smooth and
tangent to $R$ at $E_1\cap R$ and on $B$ a curve $D_B=D_B^\prime\cup
E_1$, with $D_B^\prime\sim nH-E_1$ passing through $E_1\cap R$ and
having fixed tangent direction at $E_1\cap R$.

Now, by using the notation above, what we proved implies that, if
$\gamma$ is a section of $\mathcal Y$ intersecting $F_2$ at a
general point $q$ and $\psi$ is a general section of
$\Gamma\subset\mathcal Y^1$ passing through $q=F_2\cap\Gamma$, by
denoting by $\mathcal Z^1$ the blowing up of $\mathcal Y^1$ along
$\psi$, we have that, for every divisor $S$ in the linear series
$|\mathcal O_{\mathcal Z^1}(nH-2\Gamma-3\Psi)|$, the restriction of
$S$ to the special fibre of $\mathcal Z^1$ contains the fibres $F_1$
and $F_2$ with multiplicity at least $2$. Moreover, always by using
the notation above, if $\psi\subset\Gamma\subset\mathcal Y^1$ is a
section such that the proper transform $\psi^\prime$ of $\psi$ in
$\mathcal Y^2$ intersects $\Theta_2$ at $x^1_q$ or $x^2_q$, if
$\mathcal Y_0^4$ is the special fibre of $\mathcal Y^4$ and $D$ is
the liner equivalence class of the divisor
$2\Gamma^\prime+3\Psi+2\Theta_2^{\prime\prime}+\Theta_{E_1}+\Theta_1+2T_1\subset\mathcal
Y^4,$ then \textit{the image }$W_{\gamma}$ \textit{of the
restriction map}
$$
r_0:H^0(\mathcal Y^4,\mathcal O_{\mathcal Y^4}(nH-D))\to
H^0(\mathcal Y^4_0,\mathcal O_{\mathcal Y^4_0}(nH-D))
$$
\textit{is a linear system of dimension} $dim(|\mathcal O_{A\cup
B}(nH)|)-5.$ Now, since by \eqref{divisoresutheta2}, all divisors in
$|\mathcal O_{\mathcal Y^4}(nH-D)|$ restrict to the same divisor on
$\Theta^{\prime\prime}_2$ and $T_1$, the image
$$U_\gamma\subset T_A(d,n)\times_{|\mathcal
O_R(n)|}T_{E_i}(d,n)\cap\mathcal V_0\subset |\mathcal O_{A\cup
B}(n)|$$ of $W_{\gamma}$, through the natural morphism $|\mathcal
O_{\mathcal Y^4_0}(nH-D)|\to |\mathcal O_{A\cup B}(n)|$, has still
dimension $dim(|\mathcal O_{A\cup B}(nH)|)-5.$ Moreover, we stress
that, denoting by $\tilde{C_q^i}$ and $A^{\prime\prime}$ the proper
transforms of $C_q^i$ and $A$ to $\mathcal Y^4$, by using
\eqref{divisoresutheta2} and by using that $S^4|_{T^1}$ is the fibre
$F_{q,i}$ in the ruling $|F_1|$, we have that the general divisor in
the linear system $W_\gamma$  cuts $T_1\cap A^{\prime\prime}$at the
point $y_{q,i}=F_{q,i}\cap A^{\prime\prime}$ and
$\Theta_2^{\prime\prime}\cap B^\prime$ at the point
$x_{q,i}=\tilde{C^i_q}\cap B^\prime$, different from
$E_1\cap\Theta_2^{\prime\prime}$. Now, in order to prove that
$T_A(d,n)\times_{|\mathcal O_R(n)|}T_{E_i}(d,n))$ is an irreducible
component of $\mathcal V_0$, it is enough to prove that $U_\gamma$
is not contained in any irreducible component $V$ of $\mathcal V_0$
which we have found previously in this Section. We will prove this
only for $V=V^{A}_{nH,1,0}\times_{|\mathcal O_R(n)|}|\mathcal
O_{B}(n)|$. In the other cases, you can use a similar argument. To
prove that $U_\gamma$ is not contained in $V$ it is enough to prove
that the general element of $W_\gamma$ does not corresponds to a
curve with a cusp at a point of $A^{\prime\prime}$. To see this,
observe that, by using the notation of the introduction, by the
generality of $\pi\subset\mathbb P^3$, the point $p_1\subset
S^{d-1}$ corresponding to $E_1\subset B$ is a general point of
$S^{d-1}$ and the tangent line to $R$ at $p_1$ is a general tangent
line to $S^{d-1}$ at $p_1$. This implies that the family of curves
$\mathcal V$ with a cusp on $A^{\prime\prime}$ in the linear system
$|\mathcal O_{\mathcal Y^4_0}(nH-D)|$ is irreducible of dimension $
dim(|\mathcal O_{A\cup B}(n)|)-2-3$, but it is not a linear system,
so it cannot coincide with $W_\gamma$.
\end{proof} Previous Lemmas of this sections imply the
following theorem.
\begin{theorem}\label{theoremacuspide}
Let $\mathcal V^0$ be the special fibre of $\mathcal V_{nH, 1,0}$.
Assume that $n\geq 3$ and $d\geq 2$. Then, the irreducible
components of $\mathcal V_0$ are
\begin{itemize}
\item $V^A_{nH,1,0}\times_{|\mathcal O_R(n)|}|\mathcal O_{B}(n)|$
with geometric multiplicity $1$;\\
\item $ |\mathcal O_{A}(n)|\times_{|\mathcal O_R(n)|}V^{B}_{nH,1,0}$
with geometric multiplicity $1$;\\
\item $F(d,n)$ with
 geometric multiplicity $2$;\\
\item $S_A(d,n)\times_{|\mathcal O_R(n)|}T_B(d,n)$, $T_A(d,n)\times_{|\mathcal
O_R(n)|}S_B(d,n)$, $T_A(d,n)\times_{|\mathcal O_R(n)|}T_{E_i}(d,n)$,
for $i=1,\dots,\,d(d-1)$, with
 geometric multiplicity $3$.
\end{itemize}
If $n=2$ and $d\geq 3$ the description of $\mathcal V_0$ is as in
the previous case, except for the fact that, in this case, $
|\mathcal O_{A}(n)|\times_{|\mathcal O_R(n)|}V^{B}_{nH,1,0}$ does
not appear. Finally, if $d=n=2$ then the irreducible components of
$\mathcal V_0$ are $S_A(d,n)\times_{|\mathcal O_R(n)|}T_B(d,n)$,
$\,T_A(d,n)\times_{|\mathcal O_R(n)|}S_B(d,n)$,
$T_A(d,n)\times_{|\mathcal O_R(n)|}T_{E_1}(d,n)$,
$T_A(d,n)\times_{|\mathcal O_R(n)|}T_{E_2}(d,n)$, all with geometric
multiplicity $3$.
\end{theorem}

\subsection*{Acknowledgment} I would like to express my deep
gratitude to Prof. C. Ciliberto and Prof. J. Harris for many useful
conversations I had with both of them on the subject of this paper.
I have also enjoyed and benefited from conversation with Maksym
Fedorchuk. Moreover, I would like to thank Prof. L. Caporaso and
Prof. S. Kleiman for giving me useful references at the beginning of
this work. Finally, I am grateful to the referee for helpful
comments and corrections.

\end{document}